\documentclass[11pt]{amsart}
\usepackage{graphicx}
\usepackage{array}
\newcolumntype{M}[1]{>{\centering\arraybackslash}m{#1}}
\newcolumntype{N}{@{}m{0pt}@{}}
\usepackage[margin=1in]{geometry}
\usepackage{subcaption}
\usepackage{subfiles}
\usepackage[b]{esvect}
\usepackage{amssymb}
\usepackage{comment}
\usepackage{float}
\usepackage{marginnote}
\usepackage{euscript}
\usepackage{mathdots}
\usepackage{float}
\usepackage[dvipsnames]{xcolor}		
\usepackage{mathrsfs}
\usepackage{amsthm}
\usepackage{amsmath}
\usepackage{stmaryrd}
\usepackage{tikz}
\usepackage{tikz-cd}
\usepackage{accents}
\usepackage{upgreek}
\usepackage{enumerate}
\usepackage{bm}
\usepackage{mathtools}
\usetikzlibrary{patterns}
\usepackage[all]{xy}
\usepackage{caption}
\usepackage{url}
\usepackage{float}
\usepackage{todonotes} 
\usepackage{colonequals}
\usepackage{bbm}
\usepackage{longtable}
\usepackage[full]{textcomp}
\usepackage[cal=cm]{mathalfa}
\usepackage{xparse}
\usepackage{comment}
\usepackage{enumitem}
\usetikzlibrary{fadings}
\usepackage[useregional]{datetime2}
\usepackage[
backend=biber,
style=alphabetic,
]{biblatex}
\usepackage{hyperref}
\usepackage{cleveref}
\newcommand{\x}{\boldsymbol{x}}
\newcommand{\colVec}[2]{\begin{psmallmatrix} #1 \\ #2 \end{psmallmatrix}}
\newcommand{\cone}{\mathrm{Cone}}
\newcommand{\Z}{\mathbb{Z}}

\newcommand{\R}{\mathbb{R}}
\newcommand{\C}{\mathbb{C}}

\newcommand{\bigal}{\boldsymbol{\upalpha}}
\newcommand{\alsig}{\tilde{\Sigma}_{\bigal}}

\newcommand*\circled[1]{\tikz[baseline=(char.base)]{
            \node[shape=circle,draw,inner sep=2pt] (char) {#1};}}
\newcommand{\overbar}[1]{\mkern 1.5mu\overline{\mkern-5mu#1\mkern-1.6mu}\mkern 1.5mu}
\newcommand{\vvleft}[1]{\mathchoice
  {\reflectbox{\ensuremath{\displaystyle \vv{\reflectbox{\ensuremath{\displaystyle #1}}}}}}
  {\reflectbox{\ensuremath{\textstyle \vv{\reflectbox{\ensuremath{\textstyle #1}}}}}}
  {\reflectbox{\ensuremath{\scriptstyle \vv{\reflectbox{\ensuremath{\scriptstyle #1}}}}}}
  {\reflectbox{\ensuremath{\scriptscriptstyle \vv{\reflectbox{\ensuremath{\scriptscriptstyle #1}}}}}}
}
\newcommand{\vsqsubset}{\mathrel{\text{\scalebox{1}[1.2]{$\sqsubset$}}}}
\newcommand{\sqC}{\scalebox{0.9}[1.3]{$\sqsubset$}}
\newtheorem{theorem}{Theorem}[section]

\newtheorem{corollary}[theorem]{Corollary}
\newtheorem{lemma}[theorem]{Lemma}
\newtheorem{proposition}[theorem]{Proposition}
\theoremstyle{plain}
\newtheorem{innercustomthm}{Theorem}
\newenvironment{customthm}[1]
{\renewcommand\theinnercustomthm{#1}\innercustomthm}
{\endinnercustomthm}

\crefname{innercustomthm}{theorem}{theorems}
\hypersetup{
  colorlinks   = true,          
  urlcolor     = olive,          
  linkcolor    = olive,          
  citecolor   = olive             
}

\makeatletter
\def\@tocline#1#2#3#4#5#6#7{\relax
	\ifnum #1>\c@tocdepth 
	\else
	\par \addpenalty\@secpenalty\addvspace{#2}%
	\begingroup \hyphenpenalty\@M
	\@ifempty{#4}{%
		\@tempdima\csname r@tocindent\number#1\endcsname\relax
	}{%
		\@tempdima#4\relax
	}%
	\parindent\z@ \leftskip#3\relax \advance\leftskip\@tempdima\relax
	\rightskip\@pnumwidth plus4em \parfillskip-\@pnumwidth
	#5\leavevmode\hskip-\@tempdima
	\ifcase #1
	\or\or \hskip 1em \or \hskip 2em \else \hskip 3em \fi%
	#6\nobreak\relax
	\dotfill\hbox to\@pnumwidth{\@tocpagenum{#7}}\par
	\nobreak
	\endgroup
	\fi}
\makeatother

\theoremstyle{definition}

\newtheorem{remark}[theorem]{Remark}

\newtheorem*{runningexample*}{Running example}

\newtheorem*{aside*}{Aside}

\newtheorem{construction}[theorem]{Construction}

\newtheorem{definition}[theorem]{Definition}
\newtheorem{example}[theorem]{Example}

\newtheorem{proposition-definition}[theorem]{Proposition-Definition}

\title[Cut and paste invariants of moduli spaces of stable maps to toric surfaces]{Cut and paste invariants of moduli spaces of \\ stable maps to toric surfaces}
\author{Cat Rust}
\date{\today}

\begin{document}

\maketitle
\begin{abstract}
   We study moduli spaces of logarithmic stable maps to proper toric surfaces with prescribed tangency conditions to the toric boundary. Fixing a surface, we define a chamber decomposition on the space of all tangencies such that as a function of the tangency data, the class of the corresponding moduli space in the Grothendieck ring of varieties is constant. The tangency data defines a collection of lines through the origin which, along with the rays of the toric fan, are cyclically ordered. The chambers of the decomposition are regions for which this cyclic ordering is constant and generalise the well-known resonance chambers. We pose the open question of whether the Gromov--Witten invariants of the moduli spaces are polynomial on these chambers. 
\end{abstract}

\tableofcontents
\section{Introduction}
\subsection{Main result}

Let \(X_\Sigma\) be a proper toric surface with fan \(\Sigma\) in \(N_\R\) and \(\upalpha \in N^n\). Denote by \(\mathcal{M}(X_\Sigma, \upalpha)\) the moduli space parametrising maps \[(\mathbb{P}^1, p_1, \ldots, p_n) \to (X_\Sigma, \partial X_\Sigma)\] where \(\upalpha \in N^n\) encodes the tangency of each \(p_i\) to \(\partial X_\Sigma\). For all \(\upalpha\), we have \(\mathcal{M}(X_\Sigma, \upalpha) \cong \mathcal{M}_{0,n} \times (\mathbb{C}^*)^2\) \cite[Section 4.4]{tropLog}. This space has a toroidal compactification: the moduli space of \emph{logarithmic stable maps} which we denote by \(\,\overbar{\mathcal{M}}(X_\Sigma, \upalpha)\). This space parametrises logarithmic maps from \(n\)-marked, genus 0 nodal curves to \((X_\Sigma, \partial X_\Sigma)\), again with tangency encoded by \(\upalpha\). Logarithmic maps generalise classical stable maps by recording tangency information to a fixed divisor within a logarithmic structure on the source and target. They are used widely within enumerative geometry to define logarithmic Gromov--Witten invariants. 

We give a chamber decomposition of the tangency space \(N^n\) called the \emph{\(\Sigma\)-slope decomposition} (see Section \ref{walls}) such that the following main theorem holds. 

\begin{customthm}{A}[Theorem \ref{thm: body}] \label{thm: intro} As a function of \(\upalpha\), the class \(\left[\,\overbar{\mathcal{M}}(X_\Sigma, \upalpha)\right] \in K_0(\mathrm{Var}/\mathbb{C})\) is constant on the open regions of the \(\Sigma\)-slope decomposition of \(N^n\). 
\end{customthm}

Here, \(\,\overbar{\mathcal{M}}(X_\Sigma, \upalpha)\) denotes the coarse space of the moduli space of unsaturated (see Section \ref{saturation} for a discussion of saturation) stable maps to \(X_\Sigma\). Theorem \ref{thm: intro} is proved by working with the combinatorial stratification of \(\,\overbar{\mathcal{M}}(X_\Sigma, \upalpha)\) (see Section \ref{stratif} for details). The Grothendieck ring \(K_0(\mathrm{Var}/\mathbb{C})\) of \(\C\)-varieties enriches many topological invariants.   

\subsection{Chamber structure}

We give the hypersurfaces in \(N^n\) which form the walls of the \(\Sigma\)-slope decomposition. Further details will be given in Section \ref{walls}. Note that each wall is given either by a linear or quadratic equation. 

Write \(\upalpha = (\colVec{x_1}{y_1}, \ldots, \colVec{x_n}{y_n}) \in N^n\) and let \(I \subset [n] := \{1, \ldots, n\}\) be a non-empty subset. Define \[\upalpha_{I} = \sum_{i \in I} \upalpha_i.\] The \(\upalpha_I\) are the slopes of the finite edges of the dual graphs of source curves and each defines a line through the origin of slope \(\upalpha_I\) in \(N_\mathbb{R}\). Similarly, let each ray \(\rho \in \Sigma(1)\) have primitive generator \(v_\rho = \colVec{\rho_x}{\rho_y}\). Then, again each ray defines a line of slope \(v_\rho\) in \(N_\R\)

For fixed \(\upalpha\), \(\Sigma\), there is a cyclic ordering on the subsets of \([n]\) and \(\Sigma(1)\) induced by projecting the rays of slope \(\upalpha_I\) and \(v_\rho\) onto \(S^1\) (see Figure \ref{cyclicOrder}). We define the walls to produce chambers where this cyclic ordering is constant.  

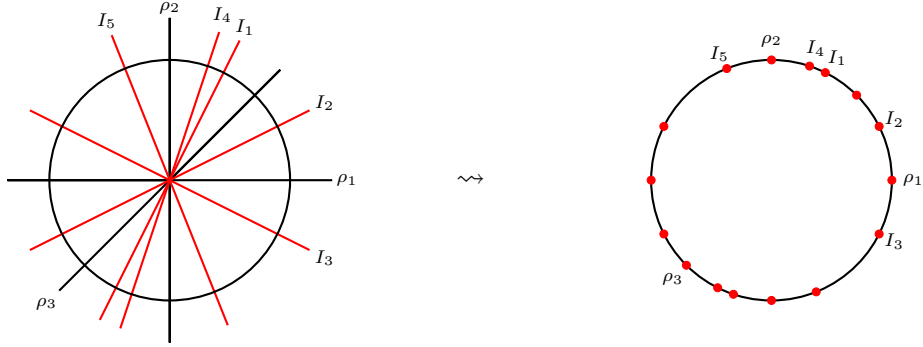
\begin{figure}
    \centering
            \begin{tikzpicture}[scale = 2]
            \draw[thick] (1.08, 0) -- (-1.08, 0) -- (0,0) -- (0, 1.08) -- (0, -1.08) -- (0,0) -- (0.73485, 0.73485) -- (-0.73485, -0.73485);
            \draw[red, thick] (0.46476, 0.92952) -- (-0.46476, -0.92952);
            \draw[red, thick] (0.32863, 0.9859) -- (-0.32863, -0.9859);
            \draw[red, thick] (-0.38596, 0.9649) -- (0.38596, -0.9649);
            \draw[red, thick] (0.92952, 0.46476) -- (-0.92952, -0.46476);
            \draw[red, thick] (0.92952, -0.46476) -- (-0.92952, 0.46476);
            \node at (0.5, 1.02) {\tiny{\(I_1\)}};
            \node at (1.03, 0.475+0.04) {\tiny{\(I_2\)}};
            \node at (1.03, -0.475-0.04) {\tiny{\(I_3\)}};
            \node at (0.35, 1.1) {\tiny{\(I_4\)}};
            \node at (-0.42, 1.06) {\tiny{\(I_5\)}};
            \node at (1.18, 0) {\tiny{\(\rho_1\)}};
            \node at (0, 1.15) {\tiny{\(\rho_2\)}};
            \node at (-0.82, -0.82) {\tiny{\(\rho_3\)}};

            \node at (2, 0) {\(\rightsquigarrow\)} ;
            \draw[thick] (0,0) circle (0.8cm);
            \draw[thick] (4,0) circle (0.8cm);
            \filldraw[red] (4, 0.8) circle (0.75pt);
            \filldraw[red] (4.8, 0) circle (0.75pt);
            \filldraw[red] (4, -0.8) circle (0.75pt);
            \filldraw[red] (3.2, 0) circle (0.75pt);
            \filldraw[red] (0.56569+4, 0.56569) circle (0.75pt);
            \filldraw[red] (-0.56569+4, -0.56569) circle (0.75pt);
            \filldraw[red] (0.8*0.31623+4, 0.8*0.94868) circle (0.75pt);
            \filldraw[red] (-0.31623*0.8+4, -0.94868*0.8) circle (0.75pt);
            \filldraw[red] (0.44721*0.8+4, 0.89443*0.8) circle (0.75pt);
            \filldraw[red] (-0.44721*0.8+4, -0.89443*0.8) circle (0.75pt);
            \filldraw[red] (-0.37139*0.8+4, 0.92848*0.8) circle (0.75pt);
            \filldraw[red] (0.37139*0.8+4, -0.92848*0.8) circle (0.75pt);
            \filldraw[red] (0.89443*0.8+4,0.44721*0.8) circle (0.75pt);
            \filldraw[red] (0.89443*0.8+4,-0.44721*0.8) circle (0.75pt);
            \filldraw[red] (-0.89443*0.8+4,0.44721*0.8) circle (0.75pt);
            \filldraw[red] (-0.89443*0.8+4,-0.44721*0.8) circle (0.75pt);

            \node at (0.8*0.31623+4+0.04, 0.8*0.94868+0.12) {\tiny{\(I_4\)}};
            \node at (0.44721*0.8+4+0.08, 0.89443*0.8+0.1) {\tiny{\(I_1\)}};
            \node at (-0.37139*0.8+4-0.05, 0.92848*0.8+0.1) {\tiny{\(I_5\)}};
            \node at (0.89443*0.8+4+0.1,0.44721*0.8+0.05) {\tiny{\(I_2\)}};
            \node at (0.89443*0.8+4+0.1,-0.44721*0.8-0.05) {\tiny{\(I_3\)}};
            \node at (-0.56569+4-0.09, -0.56569-0.09) {\tiny{\(\rho_3\)}};
            \node at (0+4, 0.93) {\tiny{\(\rho_2\)}};
            \node at (1.18*0.8+4, 0) {\tiny{\(\rho_1\)}};
        \end{tikzpicture}
    \caption{The fan of \(\mathbb{P}^2\) with rays extended to full lines is shown in black and rays of slope \(\upalpha_{I_j}\) for subsets \(I_j \subset [n]\) are shown in red. Projecting each ray onto \(S^1\) induces a cyclic ordering on the subsets of \([n]\) and the rays of \(\Sigma\).}
    \label{cyclicOrder}
\end{figure}

The cyclic ordering will change when \(\upalpha_I = k\upalpha_J\) for \(k \in \mathbb{Q}^*\) and subsets \(I\), \(J \subset [n]\), or when \(\upalpha_I\) becomes parallel to \(v_\rho\) for some \(\rho \in \Sigma(1)\). As such, the walls of the slope arrangement are defined by quadratic hypersurfaces of the form \[W_{IJ} = \left\{(\colVec{x_1}{y_1}, \ldots, \colVec{x_n}{y_n}) \in N^n : \sum_{i \in I} x_i \sum_{j \in J} y_j = \sum_{i \in I} y_i \sum_{j \in J} x_j\right\}\] and linear hypersurfaces of the form \[W_{\rho, I} = \left\{(\colVec{x_1}{y_1}, \ldots, \colVec{x_n}{y_n}) \in N^n :  \rho_x \sum_{i \in I} y_i = \rho_y \sum_{i \in I} x_i\right\}.\] The connected regions cut out by all \(W_{IJ}\) and \(W_{\rho,I
}\) give the chambers of the \(\Sigma\)-slope decomposition.

\subsection{Proof strategy}

For Theorem \ref{thm: intro}, we use that given any \(\upalpha\), \(\upalpha' \in N^n\), we have stratifications \[\overbar{\mathcal{M}}(X_\Sigma, \upalpha) = \bigsqcup_{c} \mathcal{M}(c), \quad \overbar{\mathcal{M}}(X_\Sigma, \upalpha') = \bigsqcup_{c'} \mathcal{M}(c')\] where \(\mathcal{M}(c)\), \(\mathcal{M}(c')\) are locally closed strata indexed by realisable combinatorial types \(c\), \(c'\) (see Definition \ref{combinTypeDef}). In Section \ref{formula}, given a combinatorial type \(c\), we give a formula for \([\mathcal{M}(c)]\) by considering the algebro-geometric interpretation of such types: each corresponds to the isomorphism class of a map from a nodal curve to \(X_\Sigma\) with a prescribed interaction with \(\partial X_\Sigma\). One can then use standard results in toric geometry to find the moduli of such a map.

Then, in Section \ref{construct}, given that \(\upalpha\), \(\upalpha'\) are in the same chamber of the \(\Sigma\)-slope decomposition, we build a bijection \(c \leftrightarrow c'\) such that \[[\mathcal{M}(c)] = [\mathcal{M}(c')] \in K_0(\mathrm{Var}/\mathbb{C}).\] We do this by passing to a subdivision of \(\Sigma\) which is used to build an appropriate combinatorial type matching the data of \(c\). The realisability of such a combinatorial type depends on the slopes of the edges of the dual graphs given \(\upalpha'\). When \(\upalpha'\) and \(\upalpha\) are in the same chamber, the slopes are appropriate to build the matching combinatorial type. Given the bijection, it follows that \([\,\overbar{\mathcal{M}}(X_\Sigma, \upalpha)] = [\,\overbar{\mathcal{M}}(X_\Sigma, \upalpha')] \in K_0(\mathrm{Var}/\C)\), proving Theorem \ref{thm: intro}. 

This method differs significantly from the proof of the analogous result for the space of expanded maps to \(\mathbb{P}^1\) given in \cite{kannan}. This is because the stratification of maps to a toric surface is much more intricate: for expanded maps to \(\mathbb{P}^1\), the combinatorics can be handled by considering ``left-right'' orderings (``admissible ordered partitions'') that are in bijection with combinatorial types. In two dimensions, the added degree of freedom mean there is no straightforward algorithm to enumerate strata, necessitating a constructive proof.

\subsection{Brief history and related work}

This project was inspired by the work of Siddarth Kannan on the moduli space \(\,\overbar{\mathcal{M}}(\upalpha)\) of expanded relative stable rubber maps to 1-dimensional toric targets satisfying ramification data \(\upalpha\) \cite[Theorem B]{kannan}. In this setting, Kannan proves that as a function of \(\upalpha \in \Z^n\), the class \([\,\overbar{\mathcal{M}}(\upalpha)] \in K_0(\mathrm{Var}/\mathbb{C})\) is constant on the chambers of the resonance decomposition of \(\Z^n\), first defined in \cite[Section 1]{hurw}. This project is a higher dimensional analogue of this result, working instead with unexpanded stable maps.

Kannan and Song have also used combinatorial stratifications to study invariants of moduli spaces of curves and maps in \cite{kannanTerry2, kannanTerry} and to study normal crossings compactifications of moduli spaces of maps in \cite{kannanTerry3}. Studying moduli spaces via tropicalisation has become a central technique in the field, starting with the moduli space of curves \cite{moduliOfCurves}. This framework was later generalized to study moduli spaces of admissible covers \cite{admissCovers}, weighted stable curves \cite{weightedStableCurves} and stable maps to toric varieties \cite{skeleton}. 

\subsection{Future directions}

\subsubsection{Higher dimensional targets} A natural generalisation of this problem is to consider higher dimensional toric targets. Here, the same approach and tools will apply, however the combinatorics will again increase in complexity. It is unclear by how much: the jump from 1-dimensional toric targets \cite{kannan} to 2 was sizeable. The added dimension introduces a level of geometric intricacy that makes the problem far less tractable due to the number of potential combinatorial types.

\subsubsection{Expanded stable maps} One can also ask the analogous question for spaces of expanded stable maps. In this setting, the combinatorics is much more intricate: one needs to keep track of self-intersections of curves in the image of a tropical map. It remains to be seen whether the same approach as Section \ref{collapsed} would be fruitful. 

However, one can consider the space of expanded maps to the \(k\)-dimensional logarithmic torus, \(\mathbb{G}^k_{\log}\), which has tropicalisation \(\mathbb{R}^k\). In this case, the combinatorial complexity of having maps to a fan is removed, with the only combinatorial problem being tracking self-intersections of a curve. This is work in progress. 

\subsubsection{Chamber structure for other topological invariants} One can consider other topological invariants, such as the Euler--Satake (or orbifold) Euler characteristic. This is not constant on the chambers of the \(\Sigma\)-slope decomposition since it is sensitive to finite stabilisers: see Section \ref{satake} for an example. The same is true for 1-dimensional targets for the space of maps considered in \cite{kannan}. 

\subsubsection{Chambers of polynomiality for Gromov--Witten invariants} Finally, the chambers presented in this paper are a generalisation of the resonance chamber decomposition of \(\mathbb{Z}^n\). These are chambers of polynomiality for genus 0 double-Hurwitz numbers \cite{hurw}, a result later generalized to arbitrary genus in \cite{arbGenus}.

These results raise the following question: \begin{center}
    \textbf{Q.} Are the Gromov--Witten invariants of \(\,\overbar{\mathcal{M}}(X_\Sigma, \upalpha)\) polynomial on the chambers of the \(\Sigma\)-slope decomposition?
\end{center} 

\subsection{Outline}

We begin in Section \ref{diff} by recalling the various notions of moduli spaces of stable maps, followed by a description of the combinatorial stratification of \(\,\overbar{\mathcal{M}}(X_\Sigma, \upalpha)\) in Section \ref{stratif}. Section \ref{grothRing} then provides the necessary background on the Grothendieck ring of varieties. In Section \ref{formula}, we derive a combinatorial formula for the class of a given stratum in the Grothendieck ring of varieties. Sections \ref{equiv}, \ref{fan} and \ref{lift} build the framework required for the proof of Theorem \ref{thm: intro}, which is given in Section \ref{construct}. Finally, in Section \ref{other}, we define the walls governing the \(\Sigma\)-slope decomposition and discuss alternative topological invariants of \(\,\overbar{\mathcal{M}}(X_\Sigma, \upalpha)\).

\subsection{Acknowledgements} The author is deeply grateful to Navid Nabijou for his guidance and careful attention to detail throughout this project, and for fostering a wonderfully supportive research environment at Queen Mary University of London. The author would also like to thank Adrian Sheard for many helpful conversations and Siddarth Kannan, Dhruv Ranganathan, Qaasim Shafi and Terry Song for helpful feedback on an earlier version of this work. 

\subsection{Funding} The author is supported by Queen Mary University of London's Principal’s studentship.

\section{Background}\label{background}
\subsection{Different spaces of stable maps} \label{diff}

Relative genus 0 stable maps were first considered by Gathmann for the case when \(D\) is a smooth, very ample divisor in his foundational paper \cite{gathmann}. Li then generalised this work for any smooth divisor and source curves of any genus \cite{li2001, li2}. With the help of tropical geometry, Abramovich–Chen and Gross–Siebert then constructed moduli spaces of logarithmic stable maps to a simple normal crossings pair \cite{faltings1, faltings2, grossSiebert}. For the case of a toric pair, \cite{skeleton} explicitly describes the tropicalisation to tropical maps from dual curves to fans, which is the technique this paper hinges on. 

There are several notions of stable maps. There are three main choices to make.

\subsubsection{Rubber or rigid}

For the classical notion of stable maps, two maps \(f\), \(g\colon C \to X\) are considered isomorphic if there exists \(\psi \in \mathrm{Aut}(C)\) such that \(f \circ \psi = g\). However, maps are not identified when they differ by an automorphism of the target. This notion of isomorphism is referred to as \emph{rigid} stable maps and is the approach we take for Theorem \ref{thm: body}. When our target is toric, or a toric variety bundle, we can use the dense torus to identify stable maps, giving moduli spaces of \emph{rubber} stable maps (see e.g. \cite{kannan, bundles, graberVak, marcus}).

Choosing between these notions of isomorphism changes the resulting moduli space fundamentally, affecting the dimension. Tropically, each space has a stratification by combinatorial types of tropical maps. In the rubber case, we consider combinatorial types up to translation, but not in the rigid case.

\subsubsection{Fan or void}

This is a choice of target. We can consider maps to a toric pair \((X_\Sigma, \partial X_\Sigma)\) where we encode tangency conditions to \(\partial X_\Sigma\) (see e.g. \cite{skeleton}). This is the approach taken for Theorem \ref{thm: body}. The tropicalisation of the resulting moduli space is by combinatorial types of tropical maps from \(n\)-marked directed trees to \(\Sigma\). 

Conversely, we can choose our target to be the logarithmic torus, \(\mathbb{G}_{\mathrm{log}}^k\) (see \cite{logTorus}). Every toric variety of dimension \(k\) has a morphism to \(\mathbb{G}_{\log}^k\) which is a logarithmic modification. The stratification of the space of maps to \(\mathbb{G}_{\log}^k\) is by tropical maps to \(\mathbb{R}^k\) and is somewhat easier to handle than maps to a fan: the moduli space of stable maps to a toric variety is a subdivision of the space of maps to \(\mathbb{G}_{\mathrm{log}}^k\).      

\subsubsection{Collapsed or expanded} 

As the source curves degenerate, we can allow irreducible components to be mapped completely into boundary strata. The logarithmic structures will allow us to still keep track of tangency information in these cases: our maps will still satisfy the pre-determined tangency conditions that we fix. Doing this gives rise to spaces of ``collapsed'' stable maps (referred to simply as ``logarithmic stable maps'') \cite{faltings1, faltings2, grossSiebert}. This is the approach taken in Theorem \ref{thm: body}.

In the expanded setting, curve components do not fall into boundary strata. Instead, the target is replaced by a multi-component ``expansion'' such that the degenerated curves become tangent to the boundary divisors of these new components. This was first considered for the case of a smooth divisor in work by Li \cite{li2001} and was then developed in the logarithmic setting by Kim \cite{kim}. This work was extended for simple normal crossings pairs in \cite{dhruv}, where a degeneration formula for the resulting invariants is also given.  

For a fixed target and choice of rubber or rigid maps, the space of expanded maps is a subdivision of the space of collapsed maps. See Figure \ref{rels} for a full diagram of the relationship between each space. 

In Section \ref{collapsed}, we study the space of rigid, collapsed maps to a toric pair, denoted \(\,\overbar{\mathcal{M}}(X_\Sigma, \upalpha)\). 
This work is a higher dimensional analogue of \cite[Theorem B]{kannan} for the space of rubber, expanded maps to \(\mathbb{P}^1\).

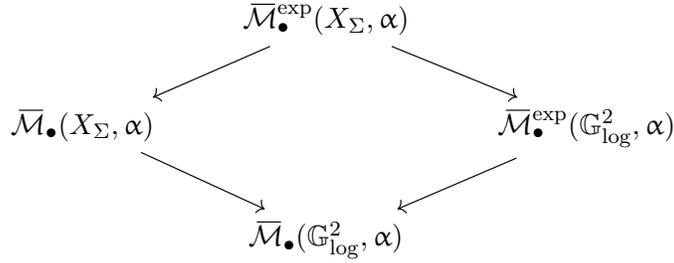
\begin{figure}
    \centering
    \begin{tikzcd}[ampersand replacement=\&]
	\& {\overbar{\mathcal{M}}^{\exp}_\bullet(X_\Sigma, \upalpha)} \\
	{\overbar{\mathcal{M}}_\bullet(X_\Sigma,\upalpha)} \&\& {\overbar{\mathcal{M}}^{\exp}_\bullet(\mathbb{G}_{\log}^2, \upalpha)} \\
	\& {\overbar{\mathcal{M}}_\bullet(\mathbb{G}_{\log}^2, \upalpha)}
	\arrow[from=1-2, to=2-1]
	\arrow[from=1-2, to=2-3]
	\arrow[from=2-1, to=3-2]
	\arrow[from=2-3, to=3-2]
\end{tikzcd}
    \caption{A diagram of relationships between spaces of stable maps. Here, \(\bullet\) denotes either rigid or rubber. Each arrow is a subdivision. Fixing the target and numerical data, there are morphisms between the rigid and rubber spaces, but these are no longer subdivisions.}
    \label{rels}
\end{figure}

\subsection{Relative maps to a toric pair}

Let \(X_\Sigma\) be a proper toric surface with complete fan \(\Sigma\) in \(N_\R\) where \(N \cong \Z^2\). Define a sub-lattice of \(N^n\) as follows: \[N^n_0 := \left\{(\upalpha_1, \ldots, \upalpha_n) \in N^n : \sum_{i=1}^n \upalpha_i = \colVec{0}{0}\right\}.\] Denote by \(\mathcal{M}(X_{\Sigma}, \upalpha)\) the moduli space of  maps \(f\colon(\mathbb{P}^1, p_1, \ldots, p_n) \to (X_\Sigma, \partial X_\Sigma)\) with contact orders given by \(\upalpha = (\upalpha_1, \ldots, \upalpha_n) \in N^n_0\) where \(\upalpha_i\) encodes the contact order between the marked point \(p_i\) and \(\partial X_\Sigma\). The condition \(\sum_i \upalpha_i = \colVec{0}{0}\) will be referred to as the \emph{balancing condition}. 
Throughout, we assume that \(\upalpha_i \neq \colVec{0}{0}\) for all \(i\). See Section \ref{zero} for a discussion of how to remove this assumption. 

To interpret each \(\upalpha_i\) as tangency data, recall that for toric varieties, there is a bijection
\[\{\text{Cartier divisors supported on } \partial X\} \longleftrightarrow \{\mathbb{Z}\text{-piecewise linear functions on } \Sigma\}.\]
Let \(\Sigma(1) = \{\rho_1, \ldots, \rho_m\}\) with primitive generators \(v_{\rho_i} \in N\) and write \(\partial X = D_1 + \ldots + D_m\). If \(X_\Sigma\) is smooth, each \(\rho_j\) corresponds to a Cartier divisor \(D_j\) and therefore there exists a piecewise linear function \(\varphi_j\colon |\Sigma| \to \mathbb{R}\) such that \(\varphi_j(v_{\rho_i}) = \delta_{ij}\). The tangency order of \(p_i\) to \(D_j\) is then defined to be \[c_{ij} = \varphi_j(\upalpha_i).\] 

As an alternative interpretation, if \(\upalpha_i\) belongs to a unique maximal cone of \(\Sigma\), for example, \(\upalpha_i \in \cone(v_{\rho_j}, v_{\rho_k})\), we can write \[\upalpha_i = c_{ij} v_{\rho_j} + c_{ik} v_{\rho_k}.\] Else, if \(\upalpha_i\) is parallel to \(v_{\rho_j}\), we have \[\upalpha_i = c_{ij} v_{\rho_j}.\] 

If \(X_\Sigma\) is singular, there may exist boundary Weil divisors which are not Cartier. However, since \(X_\Sigma\) is a surface, these divisors will be \(\mathbb{Q}\)-Cartier, and so we instead consider \(p_i\) to be tangent to \(kD_j\) (where \(kD_j\) is Cartier) with tangency orders defined in the same way as in the smooth case.

\begin{example}[Tangency Interpretation for \(X_\Sigma\) Singular]
    Consider the following fan \(\Sigma\): 

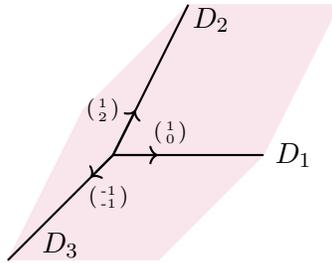
    \begin{center}
        \begin{tikzpicture}[scale = 2]
            \filldraw[purple!10] (-0.7, -0.7) -- (-0.2, 0.3) -- (0.5,1) -- (1.5,1) -- (1,0) -- (0.3, -0.7);
            \draw[thick, ->] (0,0) -- (0.3, 0);
            \node at (0.38, 0.15) {\tiny{\(\colVec{1}{0}\)}};
            \node at (-0.068, 0.3) {\tiny{\(\colVec{1}{2}\)}};
            \draw[thick, ->] (0,0) -- (0.15, 0.3);
            \draw[thick, ->] (0,0) -- (-0.15, -0.15);
            \node at (-0.03, -0.28) {\tiny{\(\colVec{\text{-}1}{\text{-}1}\)}};
            \draw[thick] (0,0) -- (1,0);
            \draw[thick] (0,0) -- (0.5, 1); 
            \draw[thick] (0,0) -- (-0.7, -0.7);
            \node at (1.2, 0) {\(D_1\)};
            \node at (0.65, 0.9) {\(D_2\)};
            \node at (-0.35, -0.6) {\(D_3\)};
        \end{tikzpicture}
    \end{center} We see that \(X_\Sigma\) is singular since \(\colVec{1}{2}\) and \(\colVec{1}{0}\) do not form a \(\mathbb{Z}\)-basis for \(N\) and therefore \(D_1\) and \(D_2\) are not Cartier. However, \(2D_1\), \(2D_2\) and \(D_3\) are Cartier, as shown by the following piecewise linear functions: 

    \begin{center}
     \resizebox{\textwidth}{!}
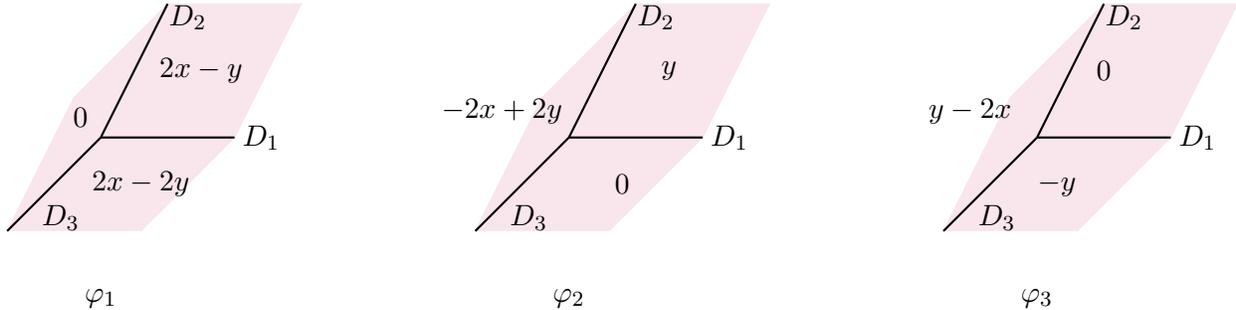
{%
        \begin{tikzpicture}[scale = 1.75]
            \filldraw[purple!10] (-0.7, -0.7) -- (-0.2, 0.3) -- (0.5,1) -- (1.5,1) -- (1,0) -- (0.3, -0.7);
            \draw[thick] (0,0) -- (1,0);
            \draw[thick] (0,0) -- (0.5, 1); 
            \draw[thick] (0,0) -- (-0.7, -0.7);
            \node at (1.2, 0) {\(D_1\)};
            \node at (0.65, 0.9) {\(D_2\)};
            \node at (-0.3, -0.6) {\(D_3\)};
            \node at (0.75, 0.5) {\(2x-y\)};
            \node at (0.3, -0.35) {\(2x-2y\)};
            \node at (-0.15, 0.15) {0};
            \node at (0, -1.2) {\(\varphi_1\)};

            \filldraw[purple!10] (-0.7+3.5, -0.7) -- (-0.2+3.5, 0.3) -- (0.5+3.5,1) -- (1.5+3.5,1) -- (1+3.5,0) -- (0.3+3.5, -0.7);
            \draw[thick] (0+3.5,0) -- (1+3.5,0);
            \draw[thick] (0+3.5,0) -- (0.5+3.5, 1); 
            \draw[thick] (0+3.5,0) -- (-0.7+3.5, -0.7);
            \node at (1.2+3.5, 0) {\(D_1\)};
            \node at (0.65+3.5, 0.9) {\(D_2\)};
            \node at (-0.3+3.5, -0.6) {\(D_3\)};
            \node at (0.5+3.75, 0.5) {\(y\)};
            \node at (0.15+3.75, -0.35) {0};
            \node at (-0.5+3.5, 0.2) {\(-2x+2y\)};
            \node at (3.5,-1.2) {\(\varphi_2\)};

            \filldraw[purple!10] (-0.7+7, -0.7) -- (-0.2+7, 0.3) -- (0.5+7,1) -- (1.5+7,1) -- (1+7,0) -- (0.3+7, -0.7);
            \draw[thick] (0+7,0) -- (1+7,0);
            \draw[thick] (0+7,0) -- (0.5+7, 1); 
            \draw[thick] (0+7,0) -- (-0.7+7, -0.7);
            \node at (1.2+7, 0) {\(D_1\)};
            \node at (0.65+7, 0.9) {\(D_2\)};
            \node at (-0.3+7, -0.6) {\(D_3\)};
            \node at (0.5+7, 0.5) {\(0\)};
            \node at (0.15+7, -0.35) {\(-y\)};
            \node at (-0.5+7, 0.2) {\(y-2x\)};
            \node at (7,-1.2) {\(\varphi_3\)};
        \end{tikzpicture}
        }
    \end{center}

    Let now \(\upalpha_i = \colVec{1}{1}\). We have 
    \begin{align*}
        c_{i1} = \varphi_1\left(\colVec{1}{1}\right) &= 1 \\ 
        c_{i2} = \varphi_2\left(\colVec{1}{1}\right) &= 1 \\ 
        c_{i3} = \varphi_3\left(\colVec{1}{1}\right) &= 0.
    \end{align*} Therefore, \(p_i\) is required to be tangent to both \(2D_1\) and \(2D_2\) with order 1 and must not meet \(D_3\). \begin{flushright}\emph{\textbf{End of example.}}\end{flushright}
\end{example}

\subsection{Logarithmic stable maps} The moduli space \(\,\overbar{\mathcal{M}}(X_\Sigma, \upalpha)\) of logarithmic stable maps compactifies \(\mathcal{M}(X_\Sigma, \upalpha)\) \cite{grossSiebert, faltings1, faltings2, skeleton}. Here, we endow \((X_\Sigma, \partial X_\Sigma)\) with the natural logarithmic structure on a toric pair. This space is a Deligne--Mumford stack \cite{faltings2, grossSiebert}. Throughout, we fix the number \(n\) of marked points and all source curves are of genus 0. The class \(\beta \in H_2(X_\Sigma, \Z)\) is determined by \(\upalpha\) since the maps are relative to the full toric boundary. 

The logarithmic structures allow us to keep track of the contact orders for components of the source curve which ``fall into'' strata of \(\partial X_\Sigma\). For instance, let \(D_i \subseteq \partial X_\Sigma\) be an irreducible component of the boundary and \(C_j \subseteq C\) an irreducible component of \(C\). Let \(C_j\), \(D_i\) be locally defined by sections \(s_{C_j}\), \(s_{D_i}\) respectively. Given \(f\colon C \to X_\Sigma\), the logarithmic structures will give \(f^\flat(s_{D_i}) = s_{C_j}^a\), indicating that \(C_j\) is mapped entirely into \(D_i\) whilst conserving the tangency information in the exponent \(a\). Note that \(a\) may be equal to zero: in this case, \(C_j\) does not fall into \(D_i\), but may still meet it at some marked point(s). 

\subsection{Saturation} \label{saturation}

For a logarithmic scheme \((X, M_X)\), the stalks of \(M_X\) are classically required to be fine and saturated. However, one can relax this condition and consider logarithmic schemes that are fine but not necessarily saturated, leading to stacks of \emph{unsaturated} stable maps, first constructed in \cite{faltings1}. 

For a moduli space \(\,\overbar{\mathcal{M}}(X)\) of logarithmic stable maps to \(X\), there exists a space of saturated maps and a space of unsaturated maps, which we denote by \(\,\overbar{\mathcal{M}}(X)^{\mathrm{sat}}\) and \(\,\overbar{\mathcal{M}}(X)^{\mathrm{unsat}}\) respectively. Strata of both spaces correspond to a monoid \(P\) defined from relations between tropical parameters for the corresponding combinatorial type. 

Given \(P\), there are processes of \emph{torification}, defined in \cite[Remark 1.2]{dan}, and pre-torification which give the local models for \(\,\overbar{\mathcal{M}}(X)^{\mathrm{sat}}\) and \(\,\overbar{\mathcal{M}}(X)^{\mathrm{unsat}}\) respectively. The torification of \(P\), \(P^{\mathrm{tor}}\), is integral, torsion-free, saturated and sharp (``toric''), whereas the pre-torification, \(P^{\mathrm{unsat}}\), is integral, torsion-free and sharp, but not necessarily saturated (``pre-toric''). 

The processes of torification and pre-torification are as follows. Begin with an arbitrary monoid \(P\).

\[
\begin{array}{c} 
\circled{1} \text{ Integralise and remove torsion} \\
\\
\left.
\begin{array}{rl}
\text{\scriptsize Construct } P^{\mathrm{unsat}} & \circled{2} \text{ Sharpen} \\[1.5em]
\color{red!80!black}\text{\scriptsize Construct } P^{\text{tor}} \text{\scriptsize  from } P^{\mathrm{unsat}} & \color{red!80!black}\circled{3} \text{ Saturate} \\[1.5em]
& \color{red!80!black}\circled{4} \text{ Sharpen}
\end{array}
\right\}
\longleftrightarrow
\left\{
\begin{array}{l}
\color{teal}\circled{2}' \text{ Saturate} \\[2em]
\color{teal}\circled{3}' \text{ Sharpen}
\end{array}
\right.
\quad
\color{teal}\begin{array}{c}\text{\scriptsize Construct}\\[-0.5ex] P^{\text{tor}}\end{array}
\end{array}
\]

Steps \(\circled{1}\), \(\circled{2}'\) and \(\circled{3}'\) give the torification process \(P \mapsto P^{\mathrm{tor}}\) as in \cite{dan}, giving the local model \(\mathrm{Spec}(\mathbb{C}[P^{\mathrm{tor}}])\) for \(\,\overbar{\mathcal{M}}(X)^{\mathrm{sat}}\). Carrying out steps 1 and 2 yields a pre-toric monoid, giving the local model \(\mathrm{Spec}(\mathbb{C}[P^{\mathrm{unsat}}])\) for \(\,\overbar{\mathcal{M}}(X)^{\mathrm{unsat}}\). Given a pre-toric monoid, carrying out steps 3 and 4 yields a toric monoid which, by \cite[Proposition 1.2.3 (v)]{ogus}, is isomorphic to \(P^{\mathrm{tor}}\). Saturating may introduce torsion, therefore when saturating after sharpening, sharpening must be carried out once more to obtain a toric monoid. 

The space of unsaturated maps will often not be normal. For \(\,\overbar{\mathcal{M}}(X)\), there is a virtual normalisation map \[\nu\colon \,\overbar{\mathcal{M}}(X)^{\mathrm{sat}} \to \,\overbar{\mathcal{M}}(X)^{\mathrm{unsat}}\] between the spaces of saturated and unsaturated maps. The map \(P^{\mathrm{unsat}} \hookrightarrow P^{\mathrm{tor}}\) gives rise to the normalisation map \[\mathrm{Spec}(\C[P^{\mathrm{tor}}]) \to \mathrm{Spec}(\C[P^{\mathrm{unsat}}]).\] Throughout, we work only with spaces of unsaturated maps and hence will drop ``\(\mathrm{unsat}\)'' from the notation. 

Now, working with the space of unsaturated maps, \(\,\overbar{\mathcal{M}}(X_\Sigma, \upalpha)\), there is a forgetful map \[\overbar{\mathcal{M}}(X_\Sigma, \upalpha) \to  \,\overbar{\mathcal{M}}_\Lambda(X_\Sigma)\] to the space of ordinary stable maps to \(X_\Sigma\) which forgets the logarithmic structures on the source and target. Here, \(\Lambda\) packages the numerical data such as the degree of the map, genus, and number of marked points. Consider the set of images of \(\mathbb{C}\)-points, \(\mathrm{Im}(X_\Sigma, \partial X_\Sigma)(\C) \subseteq \,\overbar{\mathcal{M}}_\Lambda(X_\Sigma)(\mathbb{C})\), equipped with the analytic topology. Then, there is a continuous map \[\pi\colon \overbar{\mathcal{M}}(X_\Sigma, \upalpha) \to \mathrm{Im}(X_\Sigma, \partial X_\Sigma)(\C).\] We have the following proposition. 

\begin{proposition}[{\cite[Lemma 3.7.4]{faltings1}, \cite[Proposition 3.6.3]{skeleton}}]
    Let \(\xi_1 = (C \to S, \mathcal{M}_{S,1}, f_1)\) and \(\xi_2 = (C \to S, \mathcal{M}_{S,2}, f_2)\) be minimal log stable maps such that the combinatorial types of \(\xi_1\) and \(\x_2\) coincide, and the underlying maps \(\underline{\xi_1}\), \(\underline{\xi_2}\) coincide. Then, there is a canonical isomorphism of unsaturated log maps \(\xi_1^{\mathrm{us}} \cong \xi_2^{\mathrm{us}}\).\label{3.7.4}
\end{proposition} 

\begin{corollary}
    With the analytic topology, the spaces \(\,\overbar{\mathcal{M}}(X_\Sigma, \upalpha)\) and \(\mathrm{Im}(X_\Sigma, \partial X_\Sigma)(\C)\) are homeomorphic. \label{corollary}
\end{corollary}

\begin{proof}
    Proposition \ref{3.7.4} gives that \(\pi\) is a bijection. Since \(\,\overbar{\mathcal{M}}(X_\Sigma, \upalpha)\) is compact, \(\pi\) is continuous and \(\mathrm{Im}(X_\Sigma, \partial X_\Sigma)(\C)\) is Hausdorff, \(\pi\) is a homeomorphism. Hence, we can work directly with the stratification of \(\,\overbar{\mathcal{M}}(X_\Sigma, \upalpha)\) by realisable combinatorial types of tropical maps.
\end{proof}

If we worked instead with the space of saturated maps, the proof of Theorem \ref{thm: body} as given in Section \ref{construct} would fail since stable maps corresponding to realisable combinatorial types have varying numbers of logarithmic lifts, giving that the class of \(\,\overbar{\mathcal{M}}(X_\Sigma, \partial X_\Sigma)\) in the Grothendieck is not simply a sum over strata, but a weighted sum. The same issue appears already in dimension 1, for the space of maps considered in \cite{kannan}. 

\subsection{Tropicalisation: combinatorial stratification by realisable combinatorial types}\label{stratif}

We begin with the tropicalisation of source curves.

\begin{definition}[Dual Graph, Stable Dual Graph]
    Let \((C, p_1, \ldots, p_n)\) be a genus 0, \(n\)-marked nodal curve. Define the dual graph of \(C\), \(\Gamma_C\), to be the finite, connected tree with \(n\) marked legs, labelled 1 through \(n\), such that:
    \begin{enumerate}
        \item Vertices \(v \in V(\Gamma_C)\) correspond to irreducible components \(C_v\) of \(C\).
        \item Finite edges \(e \in E(\Gamma_C)\) join vertices \(v\), \(w \in V(\Gamma_C)\) if and only if curve components \(C_v\) and \(C_w\) meet at a node.
        \item Each of the \(n\) marked legs is attached to a vertex \(v\), indicating that the corresponding marked point lies on the irreducible component \(C_v\).
    \end{enumerate}

    We call \(\Gamma_C\) \emph{stable} if for all \(v \in V(\Gamma_C)\), \(\mathrm{val}(v) \geq 3\). The set of isomorphism classes of \(n\)-marked stable trees is denoted \(\Gamma_{0,n}\). 
\end{definition}

\begin{definition}[Stabilisation]
    Let \(\Gamma \in \Gamma_{0,n}\) and \(\Gamma'\) be an arbitrary \(n\)-marked graph. We say that \(\Gamma'\) \emph{stabilises} to \(\Gamma\) if smoothing all 2-valent vertices (that is, removing each 2-valent vertex and amalgamating its two incident edges) results in a graph isomorphic to \(\Gamma\). \label{stabilise}
\end{definition}

Points in \(\,\overbar{\mathcal{M}}(X_\Sigma, \upalpha)\) parametrise two pieces of information \cite[Section 5.4]{tropLog}:

\begin{enumerate}
    \item Geometric data of a stable map from an \(n\)-pointed nodal curve \((C, p_1, \ldots, p_n) \to X_\Sigma\).
    \item Combinatorial data of an equivalence class of tropical maps from \(\Gamma_C\) to \(\Sigma\).
\end{enumerate} 
The tropical maps arise from the maps of monoids provided by the logarithmic structures. 
To define these maps, we first assign slopes to oriented edges of \(\Gamma_C\): this is the way of encoding the contact orders within the combinatorics. 

For a graph \(\Gamma\), define the set of directed edges as \[\vv{E}(\Gamma)  = \{\vv{e} = (e, v) : e \in E(\Gamma), v \in V(\Gamma), v \preceq e\}\] where \(v \preceq e\) indicates that \(e\) is incident to \(v\) and \(\vv{e}\) is directed away from \(v\). If \(e\) is a finite edge, with second vertex \(w\), define \(\vvleft{e} = (e,w)\). Similarly, we have the set of directed legs \[\vv{L}(\Gamma) = \{\vv{\ell} = (\ell, v) : \ell \in L(\Gamma), v \in V(\Gamma), v \preceq \ell\}.\] 

\begin{definition}[Balanced Slope Assignment]
    \label{balance} Let \(\upalpha \in N_0^n\) and \(\Gamma\) be a graph with directed edge set \(\vv{E}(\Gamma)\) and directed leg set \(\vv{L}(\Gamma)\). Given \(\upalpha\), we define a \emph{balanced slope assignment}, \(m_{\upalpha}\), of \(\vv{E}(\Gamma) \cup \vv{L}(\Gamma)\) to be a map 
    \begin{align*}
        m_{\upalpha} \colon \vv{E}(\Gamma) \cup \vv{L}(\Gamma) &\to N\\
        \vv{e} &\mapsto m_{\upalpha}(\vv{e})
    \end{align*}
   such that 
   \begin{enumerate}
       \item For each leg \(\vv{\ell_i} \in \vv{L}(\Gamma)\), \[m_{\upalpha}(\vv{\ell_i}) = \upalpha_i.\]
       \item For each finite edge \(\vv{e} \in \vv{E}(\Gamma)\), \[m_{\upalpha}(\vv{e}) = -m_{\upalpha}(\vvleft{e}).\] 
       \item For each \(v \in V(\Gamma)\), \[\sum_{v \preceq e} m_{\upalpha}(\vv{e}) = \colVec{0}{0}.\]
   \end{enumerate}
\end{definition}

We prove that when \(\upalpha\) satisfies the balancing condition, \(m_{\upalpha}\) is unique. 

\begin{lemma}[Lemma 3.4, \cite{rootsAndLogs}]
        Let \(\upalpha = (\upalpha_1, \ldots, \upalpha_n) \in N^n_0\) and \(\Gamma\) be a finite, \(n\)-marked tree. Then, there is a unique balanced slope assignment \(m_{\upalpha}\) of \(\vv{E}(\Gamma) \cup \vv{L}(\Gamma)\). \label{slopeLemma} 
\end{lemma}

\begin{proof}
    Assign to the leg marked \(i\) the slope \(\upalpha_i\). Then, choose a starting vertex \(v \in V(\Gamma)\) and direct all edges and legs of the graph towards \(v\). Define a level structure on \(V(\Gamma)\), \(h(v) \in \mathbb{R}\), such that if there is a directed path between \(v_1\), \(v_2 \in V(\Gamma)\), then \(h(v_1) \leq h(v_2)\). Going from the lowest to highest levels, we start at leaves of the graph. The single outgoing finite edge from a leaf has a slope determined by the slopes of the leg(s) it supports. Moving up in height, at each stage the slope of the single upwards edge is determined by the slopes of edges below, which have all already been determined. Once at \(v\), the condition that the slopes sum to \(\colVec{0}{0}\) is equivalent to the balancing condition, which is satisfied by assumption.  
\end{proof}

We now give a characterisation of the slopes of the finite edges for any \(\Gamma\). Consider \(\vv{e} = (e,v) \in \vv{E}(\Gamma)\) and \(\Gamma \setminus \vv{e}\). Since \(\Gamma\) is a tree, removing \(\vv{e}\) splits \(\Gamma\) into two halves, each supporting complementary subsets of the marked legs. Let \(I \subset [n]\) denote the set of marked legs supported by the connected component behind \(\vv{e}\) (the component containing \(v\)). To ensure that the slope assignment is balanced, we must have \[\upalpha_I + m_{\upalpha}(\vv{e}) = \colVec{0}{0}\] and hence \[m_{\upalpha}(\vv{e}) = -\upalpha_I = \upalpha_{I^{\mathrm{c}}}\] by the balancing condition. Therefore, we have the following characterisation for the slopes of the finite edges: \[\{m_{\upalpha}(\vv{e}) : \vv{e} \in \vv{E}(\Gamma)\} = \{\upalpha_I : 1 < |I| < n-1\}.\] Given any \(I \subset [n]\), there exists a stable graph \(\Gamma\) and an edge \(\vv{e} \in \vv{E}(\Gamma)\) such that \(m_{\upalpha}(\vv{e}) = \upalpha_I\).

To define a tropical map for a graph \(\Gamma\), we consider a metric enhancement, \(\sqC_\Gamma \in \mathcal{M}_{0,n}^{\mathrm{trop}}\).  

\begin{definition}[Tropical Map to \(\Sigma\)] \label{tropicalMap}
    Consider a metric graph \(\sqC \in \mathcal{M}_{0,n}^{\mathrm{trop}}\) with balanced slope assignment \(m_{\upalpha}\) for \(\upalpha \in N^n_0\) and let \(\Sigma\) be a complete fan. Identify each finite edge \(\vv{e} \in \vv{E}(\sqC)\) with \([0, l_e]\) and each leg \(\vv{\ell} \in \vv{L}(\sqC)\) with \(\mathbb{R}_{\geq 0}\). A \emph{tropical map} \(\varphi\colon\sqC \to \Sigma\) is a continuous map such that:
    
    \begin{enumerate}
        \item For each \(\vv{e} = (e, v) \in \vv{E}(\sqC)\), the restriction map \(\varphi|_{\vv{e}}\colon [0, l_e] \to \Sigma\)  has the form 
    \begin{align*}
        \varphi|_{\vv{e}}\colon [0, l_e] &\to \mathbb{R}^2\\
        t &\mapsto \varphi(v) + tm_{\upalpha}(\vv{e}).
    \end{align*}
        \item For each \(\vv{\ell} = (\ell, w) \in \vv{L}(\Gamma)\), the restriction map \(\varphi|_{\vv{\ell}}\colon \mathbb{R}_{\geq 0} \to \Sigma\)  has the form 
    \begin{align*}
        \varphi|_{\vv{\ell}}\colon \mathbb{R}_{\geq 0} &\to \mathbb{R}^2\\
        t &\mapsto \varphi(w) + tm_{\upalpha}(\vv{\ell}).
    \end{align*}
    \end{enumerate}
\end{definition} 

For the stratification of \(\,\overbar{\mathcal{M}}(X_\Sigma, \upalpha)\), we consider equivalence classes of tropical maps, captured in the following definition. We group these by stable graphs \(\Gamma \in \Gamma_{0,n}\).  

\begin{definition}[Combinatorial Type Stabilising to \(\Gamma\)]\label{ctDef}
    \label{combinTypeDef} Let \(\Gamma \in \Gamma_{0,n}\), \(\Sigma \) be a complete fan and \(\upalpha \in N_0^n\). The data of a combinatorial type \(c\) is as follows: 
    \begin{itemize}
        \item A tree \(\Gamma^c\) stabilising to \(\Gamma\) as per Definition \ref{stabilise}. 
        \item An assignment of cones \(\sigma_v\), \(\sigma_e \in \Sigma\) to each vertex \(v \in V(\Gamma^c)\) and edge or leg \(e \in E(\Gamma^c) \cup L(\Gamma^c)\) respectively such that for \(\vv{e} = (e,v)\), \(\sigma_v\) is a face of \(\sigma_e\). 
        \item A balanced slope assignment \(m_{\upalpha}\colon \vv{E}(\Gamma^c) \cup \vv{L}(\Gamma^c) \to N\) such that for \(e \in \vv{E}(\Gamma^c) \cup \vv{L}(\Gamma^c)\), \(m_{\upalpha}(\vv{e}) \in N_{\sigma_e} \subseteq N\). 

    \end{itemize}
    Further, this data must satisfy the following stability conditions:
    \begin{itemize}
        \item Each \(v \in V(\Gamma^c) \setminus V(\Gamma)\) is bivalent and \(\dim(\sigma_v) = 0\) or \(1\). 
        \item For each \(v \in V(\Gamma^c) \setminus V(\Gamma)\) supporting edges \(e_1\), \(e_2\), we have \(m_{\upalpha}(\vv{e}_1) = -m_{\upalpha}(\vv{e}_2)\) and \(\sigma_{e_1} \neq \sigma_{e_2}\). 
    \end{itemize}
    Given \(\upalpha\) and \(\Sigma\), the set of combinatorial types with dual graph stabilising to \(\Gamma \in \Gamma_{0,n}\) will be denoted \(\mathcal{C}(\Gamma, \Sigma, \upalpha)\). 
\end{definition}

A tropical map induces a combinatorial type in the following way.

\begin{construction}[Combinatorial Type of a Tropical Map]
    \label{construction} Consider \(\sqC_{\Gamma} \in \mathcal{M}_{0,n}^{\mathrm{trop}}\) enhancing \(\Gamma \in \Gamma_{0,n}\) with balanced slope assignment \(m_{\upalpha}\) for \(\upalpha \in N_0^n\). Let \(f\colon \sqC_{\Gamma} \to \Sigma\) be a tropical map. Then, \(f\) has combinatorial type \(c \in \mathcal{C}(\Gamma, \Sigma, \upalpha)\), constructed as follows:
    \begin{itemize}
        \item For each \(e \in E(\sqC_{\Gamma}) \cup L(\sqC_{\Gamma})\) such that \(f(e)\) intersects a 0 or 1-dimensional cone of \(\Sigma\), subdivide \(e\) to introduce a new bivalent vertex. All such subdivisions give a semi-stable metric graph \(\sqC_{\Gamma^c}\) with corresponding combinatorial type \(\Gamma^c\) which serves as the semi-stable graph for \(c\). The balanced slope assignment for \(\Gamma^c\) is induced from \(\Gamma\). 
        \item Considering now \(f\colon \sqC_{\Gamma^c} \to \Sigma\), for each \(v \in V(\sqC_{\Gamma^c})\), \(e \in E(\sqC_{\Gamma^c}) \cup L(\sqC_{\Gamma^c})\) set \(\sigma_v\), \(\sigma_e\) to be the smallest cones of \(\Sigma\) containing \(f(v)\), \(f(e)\) respectively. 
    \end{itemize}
\end{construction}

All examples in this paper will be for \(X_\Sigma = \mathbb{P}^2\). As such, it will be convenient to fix the following notation for the cones of \(\Sigma_{\mathbb{P}^2}\):

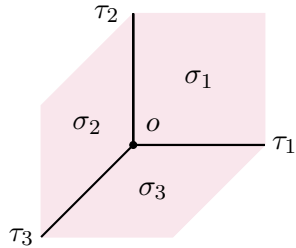
\begin{center}
    \begin{tikzpicture}[scale = 1.75]
        \filldraw[purple!10] (-0.7, -0.7) -- (-0.7, 0.3) -- (0,1) -- (1,1) -- (1,0) -- (0.3, -0.7) -- (-0.7, -0.7);
        \filldraw (0,0) circle (0.75pt);
        \draw[thick] (0,0) -- (0,1) -- (0,0) -- (1,0) -- (0,0) -- (-0.7, -0.7);
        \node at (0.5, 0.5) {\(\sigma_1\)};
        \node at (0.15, -0.35) {\(\sigma_3\)};
        \node at (-0.35, 0.15) {\(\sigma_2\)};
        \node at (1.15, 0) {\(\tau_1\)};
        \node at (-0.2, 1) {\(\tau_2\)};
        \node at (-0.85, -0.7) {\(\tau_3\)};
        \node at (0.15, 0.15) {\(o\)};
    \end{tikzpicture}
\end{center} 

\begin{example}[Tropical map and Combinatorial Type]
   Let \((C, p_1, \ldots, p_4)\) be smooth and choose \(\upalpha = \left(\colVec{1}{0}, \colVec{0}{1}, \colVec{\text{-}1}{0}, \colVec{0}{\text{-}1}\right)\). We draw \(\Gamma_C\) with the balanced slope assignment \(m_{\upalpha}\): 

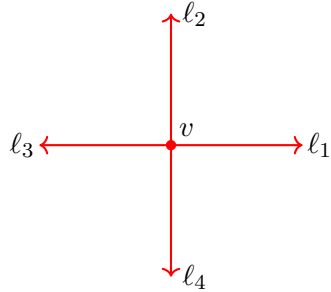
    \begin{center}
        \begin{tikzpicture}[scale = 1.75]
            \filldraw[red] (0,0) circle (1pt);
            \draw[thick, <->,red] (0,1) -- (0, -1);
            \draw[thick, <->,red] (-1,0) -- (1,0);
            \node at (0.18, 1) {\(\ell_2\)};
            \node at (1.13, 0) {\(\ell_1\)};
            \node at (-1.13, 0) {\(\ell_3\)};
            \node at (0.18, -1) {\(\ell_4\)};
            \node at (0.12, 0.12) {\(v\)};
        \end{tikzpicture}
    \end{center}

    Defining a tropical map \(\sqC_{\Gamma_C} \to \Sigma_{\mathbb{P}^2}\) amounts to choosing \(f(v)\) and drawing \(\ell_i\) with slope \(\upalpha_i\). If there are intersections of the edges of the image of \(\sqC_{\Gamma_C}\) with the 0 or 1-dimensional cones of \(\Sigma_{\mathbb{P}^2}\), we subdivide these edges, denoting the new vertices using hollow vertices. Note that this is the only way that bivalent vertices can appear: this is a stability condition on the combinatorial types. 
    
    Here are two examples of tropical maps \(\varphi\), \(\varphi'\colon \sqC_{\Gamma_C} \to \Sigma_{\mathbb{P}^2}\) such that \(f(v) \in \sigma_1\). We introduce  new bivalent vertices \(w_1\) and \(w_2\) where the image of \(\sqC_{\Gamma_C}\) intersects the sub-maximal cones of \(\Sigma\) and label the new finite edges \(e\) and \(f\).  

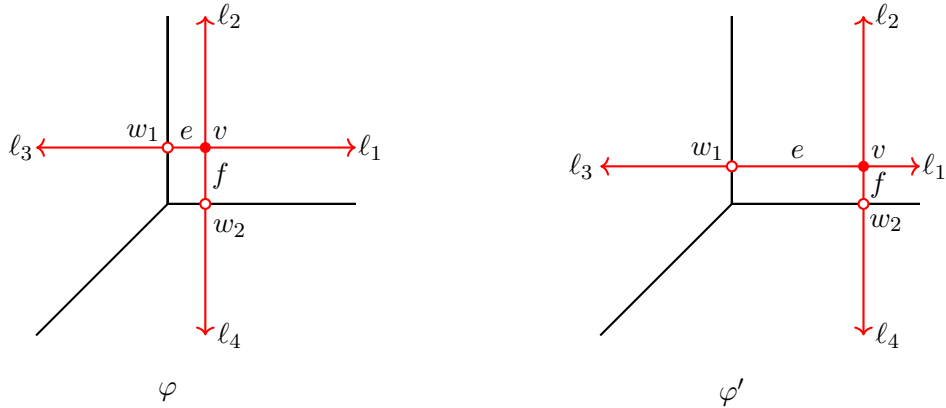
    \begin{center}
        \begin{tikzpicture}[scale = 2.5]
            \draw[thick] (0,0) -- (1,0) -- (0,0) -- (0,1) -- (0,0) -- (-0.7, -0.7);
            \draw[red, thick, <->] (0.2, 1) -- (0.2, -0.7);
            \draw[red, thick, <->] (-0.7, 0.3) -- (1, 0.3);
            \node at (0.1, 0.38) {\(e\)};
            \node at (0.28, 0.15) {\(f\)};
            \filldraw[red] (0.2, 0.3) circle (0.75pt);
            \filldraw[white] (0.2,0) circle (0.75pt);
            \draw[red, thick] (0.2, 0) circle (0.75pt);
            \filldraw[white] (0,0.3) circle (0.75pt);
            \draw[red, thick] (0, 0.3) circle (0.75pt);
            \node at (0.28, 0.38) {\(v\)};
            \node at (1.08, 0.3) {\(\ell_1\)};
            \node at (0.33, 1) {\(\ell_2\)};
            \node at (-0.78, 0.3) {\(\ell_3\)};
            \node at (0.33, -0.7) {\(\ell_4\)};
            \node at (-0.13, 0.38) {\(w_1\)};
            \node at (0.33, -0.13) {\(w_2\)};
            \node at (0, -1) {\(\varphi\)};
            \node at (3, -1) {\(\varphi'\)};

            \draw[thick] (0+3,0) -- (1+3,0) -- (0+3,0) -- (0+3,1) -- (0+3,0) -- (-0.7+3, -0.7);
            \filldraw[red] (0.7+3, 0.2) circle (0.75pt);
            \draw[red, thick, <->] (0.7+3, 1) -- (0.7+3, -0.7);
            \draw[red, thick, <->] (-0.7+3, 0.2) -- (1+3, 0.2);
            \filldraw[white] (3,0.2) circle (0.75pt);
            \draw[red, thick] (3, 0.2) circle (0.75pt);
            \filldraw[white] (3.7, 0) circle (0.75pt);
            \draw[red, thick] (3.7, 0) circle (0.75pt);
            \node at (1.08+3, 0.2) {\(\ell_1\)};
            \node at (-0.8+3, 0.2) {\(\ell_3\)};
            \node at (0.7+0.13+3, -0.7) {\(\ell_4\)};
            \node at (0.7+0.13+3, 1) {\(\ell_2\)};
            \node at (0.7+3+0.08, 0.2+0.08) {\(v\)};
            \node at (0.7 + 3 + 0.08, 0.1) {\(f\)};
            \node at (-0.13 + 3, 0.15+0.13) {\(w_1\)};
            \node at (0.7+3+0.12, -0.1) {\(w_2\)};
            \node at (0.35+3, 0.28) {\(e\)};
        \end{tikzpicture}
    \end{center} Combinatorially, \(\mathrm{Im}(\varphi)\) and \(\mathrm{Im}(\varphi')\) are indistinguishable: changing the scale horizontally or vertically will cause them to appear identical. Equivalently, in \(\mathrm{Im}(\varphi)\), one can make \(e\) longer and \(f\) shorter to match \(\mathrm{Im}(\varphi')\). This process will not change the cones of \(\Sigma_{\mathbb{P}^2}\) that the vertices (stable or unstable) are mapped to.

    Following Construction \ref{construction}, we give the combinatorial types of maps \(\varphi\), \(\varphi'\). In each case, the semi-stable graph looks as follows: 
    
    
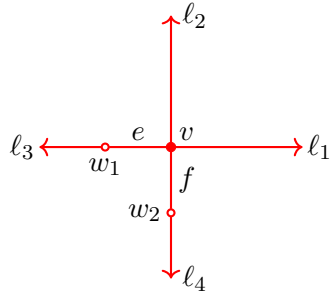
\begin{center}
        \begin{tikzpicture}[scale = 1.75]
            \filldraw[red] (0,0) circle (1pt);
            \draw[thick, <->,red] (0,1) -- (0, -1);
            \draw[thick, <->,red] (-1,0) -- (1,0);
            \node at (0.18, 1) {\(\ell_2\)};
            \node at (1.13, 0) {\(\ell_1\)};
            \node at (-1.13, 0) {\(\ell_3\)};
            \node at (0.18, -1) {\(\ell_4\)};
            \node at (0.12, 0.1) {\(v\)};
            \filldraw[white] (0,-0.5) circle (0.75pt);
            \draw[red,thick] (0, -0.5) circle (0.75pt);
            \filldraw[white] (-0.5,0) circle (0.75pt);
            \draw[red,thick] (-0.5, 0) circle (0.75pt);
            \node at (-0.25, 0.1) {\(e\)};
            \node at (0.12, -0.25) {\(f\)};
            \node at (-0.2, -0.5) {\(w_2\)};
            \node at (-0.5, -0.15) {\(w_1\)};
        \end{tikzpicture}
    \end{center} 
    Then, for both maps, the cone assignments of vertices are \[\sigma_v = \sigma_1, \quad \sigma_{w_1} = \tau_2, \quad \sigma_{w_2} = \tau_1\] and the assignments of finite edges and legs are 
    \[\text{\(\sigma_{\ell_1}\), \(\sigma_{\ell_2}\), \(\sigma_e\), \(\sigma_f = \sigma_1\),} \quad \sigma_{\ell_3} = \sigma_2, \quad \sigma_{\ell_4} = \sigma_3.\] Therefore, by Definition \ref{combinTypeDef}, \(\varphi\) and \(\varphi'\) have the same combinatorial type. 
    
    Consider now \(f(v) \in \sigma_2\). We have the following tropical maps \(\psi_i\) which induce new bivalent vertices \(s_1\), \(s_2\) and new finite edges \(g\), \(h\): 

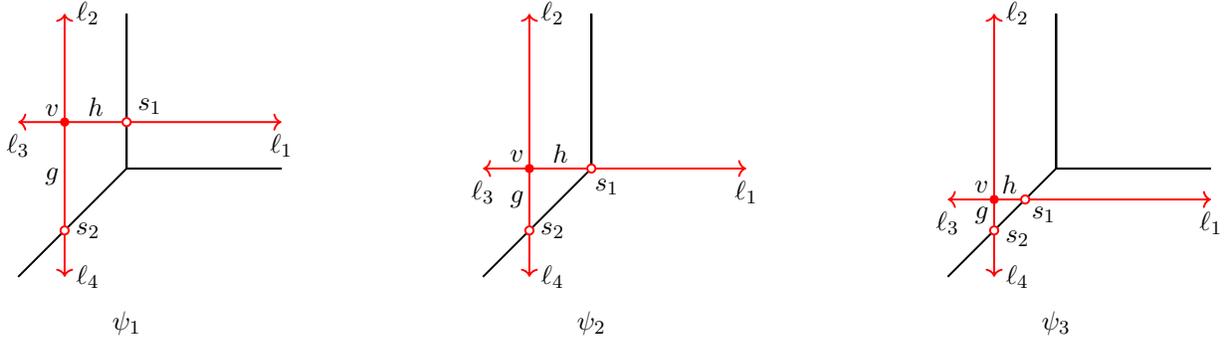
    \begin{center}
    \resizebox{\textwidth}{!}{%
        \begin{tikzpicture}[scale = 2.25]
            \draw[thick] (0,0) -- (1,0) -- (0,0) -- (0,1) -- (0,0) -- (-0.7, -0.7);
            \filldraw[red] (-0.4, 0.3) circle (0.75pt);
            \draw[red, thick, <->] (-0.4, 1) -- (-0.4, -0.7);
            \draw[red, thick, <->] (-0.7, 0.3) -- (1, 0.3);
            \filldraw[white] (0,0.3) circle (0.75pt);
            \draw[red, thick] (0, 0.3) circle (0.75pt);
            \filldraw[white] (-0.4,-0.4) circle (0.75pt);
            \draw[red, thick] (-0.4, -0.4) circle (0.75pt);
            \node at (-0.48, -0.05) {\(g\)};
            \node at (1, -0.15 +0.3) {\(\ell_1\)};
            \node at (-0.4 + 0.15, 1) {\(\ell_2\)};
            \node at (-0.4 + 0.15, -0.7) {\(\ell_4\)};
            \node at (-0.7, -0.15+0.3) {\(\ell_3\)};
            \node at (-0.2, 0.4) {\(h\)};
            \node at (-0.48, 0.38) {\(v\)};
            \node at (0.15, 0.4) {\(s_1\)};
            \node at (-0.25, -0.4) {\(s_2\)};
            \node at (0, -1) {\(\psi_1\)};
            
            \draw[thick] (0+3,0) -- (0+3,0) -- (0+3,1) -- (0+3,0) -- (-0.7+3, -0.7);
            \filldraw[red] (-0.4+3, 0) circle (0.75pt);
            \draw[red, thick, <->] (-0.4+3, 1) -- (-0.4+3, -0.7);
            \draw[red, thick, <->] (-0.7+3, 0) -- (1+3, 0);
            \filldraw[white] (3,0) circle (0.75pt);
            \draw[red, thick] (3, 0) circle (0.75pt);
            \filldraw[white] (2.6,-0.4) circle (0.75pt);
            \draw[red, thick] (2.6, -0.4) circle (0.75pt);
            \node at (-0.25+3, -0.4) {\(s_2\)};
            \node at (-0.4 + 0.15+3, 1) {\(\ell_2\)};
            \node at (-0.4 + 0.15+3, -0.7) {\(\ell_4\)};
            \node at (3+0.1, -0.12) {\(s_1\)};
            \node at (1+3, -0.15) {\(\ell_1\)};
            \node at (-0.7+3, -0.15) {\(\ell_3\)};
            \node at (-0.48+3, 0.1-0.02) {\(v\)};
            \node at (-0.2 + 3, 0.1) {\(h\)};
            \node at (-0.48 + 3, -0.2) {\(g\)};
            \node at (3, -1) {\(\psi_2\)};
            
            \draw[thick] (0+6,0) -- (1+6,0) -- (0+6,0) -- (0+6,1) -- (0+6,0) -- (-0.7+6, -0.7);
            \filldraw[red] (-0.4+6, -0.2) circle (0.75pt);
            \draw[red, thick, <->] (-0.4+6, 1) -- (-0.4+6, -0.7);
            \draw[red, thick, <->] (-0.7+6, -0.2) -- (1+6, -0.2);
            \filldraw[white] (-0.4+6,-0.4) circle (0.75pt);
            \draw[red, thick] (-0.4+6, -0.4) circle (0.75pt);
            \filldraw[white] (-0.2+6,-0.2) circle (0.75pt);
            \draw[red, thick] (-0.2+6, -0.2) circle (0.75pt);
            \node at (-0.25+6, -0.45) {\(s_2\)};
            \node at (-0.4 + 0.15+6, 1) {\(\ell_2\)};
            \node at (-0.4 + 0.15+6, -0.7) {\(\ell_4\)};
            \node at (1+6, -0.15 -0.2) {\(\ell_1\)};
            \node at (-0.7+6, -0.15-0.2) {\(\ell_3\)};
            \node at (-0.48+6, -0.1-0.02) {\(v\)};
            \node at (-0.2+0.12+6, -0.3) {\(s_1\)};
            \node at (-0.3+6,-0.1) {\(h\)}; 
            \node at (-0.48+6,-0.3) {\(g\)}; 
            \node at (6, -1) {\(\psi_3\)};
        \end{tikzpicture}
        }
    \end{center}
    Decreasing the length of \(g\) takes us from left to right. For all \(\psi_i\), the semi-stable graph looks as follows: 

    
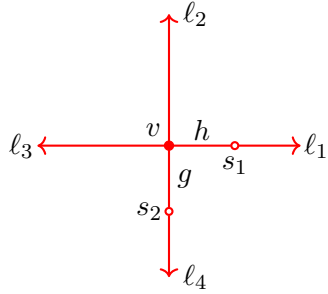
\begin{center}
        \begin{tikzpicture}[scale = 1.75]
            \filldraw[red] (0,0) circle (1pt);
            \draw[thick, <->,red] (0,1) -- (0, -1);
            \draw[thick, <->,red] (-1,0) -- (1,0);
            \node at (0.2, 1) {\(\ell_2\)};
            \node at (1.12, 0) {\(\ell_1\)};
            \node at (-1.12, 0) {\(\ell_3\)};
            \node at (0.2, -1) {\(\ell_4\)};
            \filldraw[white] (0,-0.5) circle (0.75pt);
            \draw[red,thick] (0, -0.5) circle (0.75pt);
            \filldraw[white] (0.5,0) circle (0.75pt);
            \draw[red,thick] (0.5, 0) circle (0.75pt);
            \node at (-0.12, 0.12) {\(v\)};
            \node at (-0.15, -0.5) {\(s_2\)};
            \node at (0.5, -0.15) {\(s_1\)};
            \node at (0.12, -0.25) {\(g\)};
            \node at (0.25, 0.12) {\(h\)};
        \end{tikzpicture}
    \end{center} For each \(\psi_i\), the cone assignment for \(s_1\) is different. We have \(\sigma_{s_1} = \tau_2\) for \(\psi_1\), \(\sigma_{s_1} = o\) for \(\psi_2\) and \(\sigma_{s_1} = \tau_3\) for \(\psi_3\). Therefore, by Definition \ref{combinTypeDef}, all of the \(\psi_i\) have different combinatorial types. \begin{flushright} \textbf{\emph{End of example.}}\end{flushright}
\end{example}

\subsection{Realisability} The stratification of \(\,\overbar{\mathcal{M}}(X_\Sigma, \upalpha)\) is given by \emph{realisable} combinatorial types.  

\begin{definition}[Realisable Combinatorial Type]\label{realCTdef}
    Let \(\Gamma \in \Gamma_{0,n}\) and choose \(c \in \mathcal{C}(\Gamma, \Sigma, \upalpha)\) with semi-stable graph \(\Gamma^c\). Then, \(c\) is \emph{realisable} if there exists a metric enhancement \(\sqC_{\Gamma^c}\) of \(\Gamma^c\) and a tropical map \(f\colon \sqC_{\Gamma^c} \to \Sigma\) such that the combinatorial type of \(f\), as per Construction \ref{construction}, is \(c\). 
\end{definition}

The combinatorial types we have seen so far have been induced by tropical maps, and so have been realisable. We now give an example of a non-realisable combinatorial type. 

\begin{example}[Non-Realisable Combinatorial Type] 
    Let \(\upalpha = (\colVec{1}{0}, \colVec{0}{1}, \colVec{-1}{0}, \colVec{0}{-1})\) and choose \(C = \substack{\\ \begin{tikzpicture} \draw (0, 0) -- (0.6, 0.6);
        \draw (1,0) -- (0.4, 0.6);
        \filldraw (0.1,0.1) circle (1pt);
        \filldraw (0.3,0.3) circle (1pt);
        \filldraw (0.9,0.1) circle (1pt);
        \filldraw (0.7,0.3) circle (1pt);
        \node at (-0.1, 0.1) {\tiny{1}};
        \node at (0.1, 0.3) {\tiny{2}};
        \node at (0.9, 0.3) {\tiny{3}};
        \node at (1.1, 0.1) {\tiny{4}}; \end{tikzpicture}} \in \,\overbar{\mathcal{M}}_{0,4}\). The balanced slope assignment \(m_{\upalpha}\) for \(\Gamma_C\) looks as follows:

    
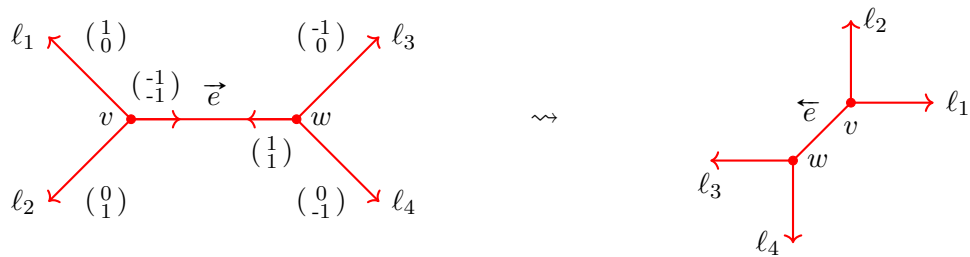
\begin{center}
        \begin{tikzpicture}[scale = 2.2]
            \filldraw[red] (0,0) circle (0.75pt);
            \draw[thick,red](0,0) -- (1,0);
            \filldraw[red] (1,0) circle (0.75pt);
            \draw[thick, ->,red] (0,0) -- (-0.5, 0.5); 
            \draw[thick, ->,red] (0,0) -- (-0.5, -0.5);
            \draw[thick, ->,red] (1,0) -- (1.5, 0.5);
            \draw[thick, ->,red] (1,0) -- (1.5, -0.5);
            \node at (1.15, 0.5) {\(\colVec{\text{-}1}{0}\)}; 
            \node at (1.15, -0.5) {\(\colVec{0}{\text{-}1}\)}; 
            \node at (-0.15, 0.5) {\(\colVec{1}{0}\)}; 
            \node at (-0.15, -0.5) {\(\colVec{0}{1}\)}; 
            \node at (0.15, 0.2) {\(\colVec{\text{-}1}{\text{-}1}\)};
            \draw[red, thick, ->] (0,0) -- (0.3,0);
            \draw[red, thick, ->] (1,0) -- (0.7,0);
            \node at (0.85, -0.2) {\(\colVec{1}{1}\)};
            \node at (2.5,0) {\(\rightsquigarrow\)};
            \node at (-0.65, 0.5) {\(\ell_1\)};
            \node at (-0.65, -0.5) {\(\ell_2\)};
            \node at (1.65, 0.5) {\(\ell_3\)};
            \node at (1.65, -0.5) {\(\ell_4\)};
            \node at (-0.15, 0) {\(v\)};
            \node at (1.15, 0) {\(w\)};
            \node at (0.5, 0.15) {\(\vv{e}\)};
            \filldraw[red] (4,-0.25) circle (0.75pt);
            \draw[thick, ->,red] (4,-0.25) -- (3.5,-0.25);
            \draw[thick, ->,red] (4,-0.25) -- (4, -0.5-0.25);
            \draw[thick,red] (4,-0.25) -- (4.35, 0.35-0.25);
            \filldraw[red] (4.35, 0.35-0.25) circle (0.75pt);
            \draw[thick, ->,red] (4.35, 0.35-0.25) -- (4.85, 0.35-0.25);
            \draw[thick, ->,red] (4.35, 0.35-0.25) -- (4.35, 0.85-0.25);
            \node at (5, 0.35-0.25) {\(\ell_1\)};
            \node at (4.5, 0.85-0.25) {\(\ell_2\)};
            \node at (3.5,-0.4) {\(\ell_3\)};
            \node at (3.85, -0.5-0.25) {\(\ell_4\)};
            \node at (4.15, -0.25) {\(w\)};
            \node at (4.35, 0.35-0.4) {\(v\)};
            \node at (4.25-0.15, -0.15+0.2) {\(\vvleft{e}\)}; 
        \end{tikzpicture}
    \end{center}

    Consider \(c \in \mathcal{C}(\Gamma_C, \Sigma, \upalpha)\) with the following cone assignments: 
    
        \begin{table}[H]
            \centering
            \begin{tabular}{|c|c|}
                \hline
                Vertex/Edge & Cone \\ \hline
                \(v\) & \(\sigma_3\) \\
                \(w\) & \(\tau_3\) \\ 
                \(\ell_i\), \(e\) & \(\sigma_1\) \\ \hline
            \end{tabular}
        \end{table}
    
    Suppose \(c\) were realisable. Then, there would exist a tropical map \(f\colon \sqC_{\Gamma_C} \to \Sigma\) such that \(f(v) \in \sigma_3^\circ\). But, to have \(f(w) \in \tau^\circ_3\), \(\vv{e}\) must cross \(\tau_3\). However, since \(m_{\upalpha}(\vv{e})\) and \(m_{\upalpha}(\tau_3)\) are parallel, this cannot happen. Therefore, \(c\) is not realisable. \begin{flushright}\emph{\textbf{End of example.}}\end{flushright}
    \label{ct}
\end{example}

We conclude this subsection with the following lemma. 

\begin{lemma}[Strata: Collapsed Maps]
    The strata of \(\,\overbar{\mathcal{M}}(X_\Sigma, \upalpha)\) are indexed by realisable combinatorial types of maps from genus 0 metric curves to \(\Sigma\).
\end{lemma}

\begin{proof}
    See, for example, \cite[Theorem B]{skeleton}.
\end{proof}

\begin{remark}
    There are various notions of realisability in the literature. A given combinatorial type \(c\) is said to be realisable if it is the combinatorial type of a logarithmic stable map to \(X_\Sigma\). If instead it is the combinatorial type of a map to the Artin fan \(\mathcal{A}_\Sigma\), \(c\) is said to be \emph{weakly realisable} (see \cite[Proposition 1.12]{dan}. Note that the terminology ``representable'' is used in place of ``realisable''). However, when source curves are of genus 0, the two notions are equivalent \cite[Theorem B]{skeleton}. As such, moving forward, we will refer to these combinatorial types simply as ``realisable''. 
\end{remark}

\subsection{The Grothendieck ring}\label{grothRing}

Let \(k\) be a field and denote by \(\mathrm{Var}/k\) the category of \(k\)-varieties.

\begin{definition}[{See e.g. \cite[Definition 2.1]{nereeja}}]
    The \emph{Grothendieck Ring}, \(K_0(\mathrm{Var}/k)\), is the quotient of the free abelian group generated by isomorphism classes \([\cdot]\) of \(k\)-varieties by the relation 
    \[[X \setminus Y ] = [X] - [Y]\] where \(Y \subseteq X\) is a closed subscheme. The fibre product over \(k\) induces a ring structure defined by \[[X] \cdot [Y] = [(X \times_k Y)_{\mathrm{red}}].\]
\end{definition}

The main proof of this paper (see Section \ref{construct}) hinges on the following property of the Grothendieck ring.

\begin{proposition}[{See e.g. \cite[Lemma 2.2]{nereeja}}]
        If a variety \(X\) is a finite disjoint union of locally closed subvarieties \[X = \bigsqcup_{i = 1}^n X_i\] then \[[X] = \sum_{i = 1}^n [X_i].\]\label{splitter}
\end{proposition}

\begin{proof}
    See \cite[Lemma 2.2]{nereeja}. 
\end{proof} This property will be used in the proof of Theorem \ref{thm: body}. 

\section{Collapsed stable maps to a toric surface} \label{collapsed}

Fix a proper toric surface \(X_\Sigma\) with fan \(\Sigma\) in \(N_\R\) and consider the moduli space of unsaturated (see Section \ref{saturation} for a discussion of saturation) logarithmic stable maps \(\,\overbar{\mathcal{M}}(X_\Sigma, \upalpha)\). As \(\upalpha\) varies in \(N_0^n\), the topology of \(\,\overbar{\mathcal{M}}(X_\Sigma, \upalpha)\) changes. Following from work of Kannan, \cite[Theorem B]{kannan}, we aim to answer the following question: 

\begin{center}
    \(\mathrm{\boldsymbol{Q^n:}}\) As a function of \(\upalpha \in N^n_0\), where is \(\left[ \,\overbar{\mathcal{M}}(X_\Sigma, \upalpha)\right] \in K_0(\mathrm{Var}/\mathbb{C})\) constant? 
\end{center} We use the stratification of \(\,\overbar{\mathcal{M}}(X_\Sigma, \upalpha)\) by realisable combinatorial types (see Section \ref{stratif} for details of the stratification) to answer this question, giving Theorem \ref{thm: body}. 

The outline of this section is as follows. In Section \ref{formula}, given a combinatorial type \(c\), we give a formula \eqref{form} for the class in the Grothendieck ring of the corresponding locally closed stratum \(\mathcal{M}(c) \subseteq \,\overbar{\mathcal{M}}(X_\Sigma, \upalpha)\). Then, in Section \ref{equiv}, we define a notion of equivalence called \emph{\(\Sigma\)-equivalence} (see Definition \ref{sigmaEquiv}) for tangency data such that two lists of tangencies are sigma equivalent if and only if they are in the same chamber of the \(\Sigma\)-slope decomposition.  

We then build up to the proof. Given \(\upalpha\), \(\upalpha'\) in the same chamber of the \(\Sigma\)-slope decomposition, to build the bijection of strata of \(\,\overbar{\mathcal{M}}(X_\Sigma, \upalpha)\) and \(\,\overbar{\mathcal{M}}(X_\Sigma, \upalpha')\), we take a combinatorial type \(c\) for \(\,\overbar{\mathcal{M}}(X_\Sigma, \upalpha)\) and build a realisable combinatorial type \(c'\) for \(\,\overbar{\mathcal{M}}(X_\Sigma, \upalpha')\) matching the combinatorial data of \(c\), thus giving \([\mathcal{M}(c)] = [\mathcal{M}(c')] \in K_0(\mathrm{Var}/\C)\). We do this in three steps. 

\begin{enumerate}
    \item In Section \ref{fan}, we define a subdivision of \(\Sigma\) depending on \(\upalpha\), denoted \(\tilde{\Sigma}_{\upalpha}\) (see Definition \ref{ssFan}). Given this, in Section \ref{lift}, we explain how to lift \(c\) to a realisable combinatorial type \(\tilde{c}\) for the space \(\,\overbar{\mathcal{M}}(X_{\tilde{\Sigma}_{\upalpha}}, \upalpha)\). 
    \item Given \(\tilde{c}\), in Section \ref{construct}, given that \(\upalpha\), \(\upalpha'\) are in the same chamber of the \(\Sigma\)-slope decomposition, we explain how to construct a tropical map to \(\tilde{\Sigma}_{\upalpha'}\) matching the data of \(\tilde{c}\). We then take the combinatorial type of this map to produce a combinatorial type \(\tilde{c}'\) for \(\,\overbar{\mathcal{M}}(X_{\tilde{\Sigma}_{\upalpha'}}, \upalpha')\). 
    \item We finally apply a forgetful map \(\,\overbar{\mathcal{M}}(X_{\tilde{\Sigma}_{\upalpha'}}, \upalpha') \to \,\overbar{\mathcal{M}}(X_\Sigma, \upalpha')\) to obtain a combinatorial type \(c'\) which matches the combinatorial data of \(c\), giving that \([\mathcal{M}(c)] = [\mathcal{M}(c')] \in K_0(\mathrm{Var}/\mathbb{C})\).  
\end{enumerate} This establishes a bijection \(c \leftrightarrow c'\) between the strata of \(\,\overbar{\mathcal{M}}(X_{\Sigma}, \upalpha)\) and \(\,\overbar{\mathcal{M}}(X_{\Sigma}, \upalpha')\) such that \[[\mathcal{M}(c)] = [\mathcal{M}(c')] \in K_0(\mathrm{Var}/\mathbb{C}).\] It then follows from Lemma \ref{splitter} that \([\,\overbar{\mathcal{M}}(X_\Sigma, \upalpha)] = [\,\overbar{\mathcal{M}}(X_\Sigma, \upalpha')] \in K_0(\mathrm{Var}/\C)\), proving Theorem \ref{thm: body}. 

Finally, in Section \ref{other}, we define the walls that give the \(\Sigma\)-slope decomposition and discuss alternative topological invariants of \(\,\overbar{\mathcal{M}}(X_\Sigma, \upalpha)\) one can consider.

    \subsection{Strata formula}\label{formula}

Given a combinatorial type \(c \in \mathcal{C}(\Gamma, \Sigma, \upalpha)\) (see Definition \ref{combinTypeDef}) denote the locally closed strata of maps with combinatorial type \(c\) by \(\mathcal{M}(c) \subseteq \,\overbar{\mathcal{M}}(X_\Sigma, \upalpha)\). 

\begin{lemma}
     Let \(c \in \mathcal{C}(\Gamma, \Sigma, \upalpha)\) have semi-stable graph \(\Gamma^c\). Define \[R_c = \sum_{e \in E(\Gamma^c)}\dim(\sigma_e) - \sum_{v \in V(\Gamma^c)} \dim(\sigma_v) + 2 - |V(\Gamma^c) \setminus V(\Gamma)|.\] Then, 
     \begin{equation}
         [\mathcal{M}(c)] = \left[\left(\prod_{v \in V(\Gamma)} \mathcal{M}_{0, \mathrm{val}(v)}\right)  \times (\mathbb{C}^*)^{R_c}\right] \in K_0(\mathrm{Var}/\mathbb{C}).\label{form}
     \end{equation}

\end{lemma}

\begin{proof}
    By Corollary \ref{corollary}, it is sufficient to work with the locus of ordinary stable maps. That is, we can ignore the data of the logarithmic structures and work with the moduli of constants and source curves. 
    
    For vertices \(v \in V(\Gamma)\), mapping \(v \mapsto \sigma_v\) is equivalent to mapping a genus 0 curve \(C_v\) into a toric variety \(X_{\sigma_v}\) of dimension \(2-\dim(\sigma_v)\) with prescribed tangency conditions to \(\partial X_{\sigma_v}\). Such maps are defined up to the dense torus of \(X_{\sigma_v}\), \(T_{X_{\sigma_v}} = (\mathbb{C}^*)^{2 - \dim(\sigma_v)}\). For stable vertices, the corresponding curve components have no non-trivial automorphisms, and so we cannot identify maps which differ by this torus action. However, for semi-stable components with 2 special points, corresponding to bivalent vertices, the subgroup of automorphisms fixing those points is isomorphic to \(\mathbb{C}^*\), giving that maps \(C_v \to X_{\sigma_v}\) are defined up to \((\mathbb{C}^*)^{2-\dim(\sigma_v) -1}\). 
    
    For \(v\) stable, each map \(C_v \to X_{\sigma_v}\) carries the moduli of the source curve, \(\mathcal{M}_{0, \text{val}(v)}\). Thus, for now, the class of \(\mathcal{M}(c)\) is \[[\mathcal{M}(c)] = \left[\left(\prod_{v \in V(\Gamma)} \mathcal{M}_{0,\text{val}(v)} \times (\mathbb{C}^*)^{2-\dim(\sigma_v)}\right) \times \prod_{v \in V(\Gamma^c) \setminus V(\Gamma)} (\mathbb{C}^*)^{2-\dim(\sigma_v) - 1}\right].\]

    Consider now \(E(\Gamma^c)\). Let \(e \in E(\Gamma^c)\) join vertices \(v\), \(w \in V(\Gamma^c)\) and consider \(\sigma_e\). If \(\dim(\sigma_e) = 0\), we have two curve components \(C_v\), \(C_w\) both mapping into the dense torus of \(X_\Sigma\) but meeting at a node \(q\) represented by \(e\). The moduli of both \(f_v \colon C_v \to T_{X_\Sigma}\), \(f_w \colon C_w \mapsto T_{X_\Sigma}\) is isomorphic to \((\mathbb{C}^*)^2\), however imposing \(f_v(q) = f_w(q)\) fixes one pair of the parameters, meaning that \(e\) causes an identification killing \((\mathbb{C}^*)^2\). 
    
    Similarly, if \(\dim(\sigma_e) = 1\), the maps from \(C_v\), \(C_w\) to the corresponding toric varieties are defined up to either 2 or 3 parameters. Again, \(e\) represents a node \(q\) which causes one of these parameters to be dependent on the other, meaning that \(e\) causes an identification killing \(\mathbb{C}^*\).
    
    Finally, if \(\dim(\sigma_e) = 2\), without loss of generality, \(C_w\) maps to a torus fixed point. Then, having \(C_v \mapsto X_{\sigma_v}\) and \(C_w \mapsto X_{\sigma_w}\) agree at \(q\) is automatic. As such, no moduli is lost. 

    Combining the above, we have the following expression for \([\mathcal{M}(c)]\): \[\left(\prod_{v \in V(\Gamma)} \mathcal{M}_{0,\text{val}(v)} \times (\mathbb{C}^*)^{2-\dim(\sigma_v)}\right) \times \prod_{v \in V(\Gamma^c) \setminus V(\Gamma)} (\mathbb{C}^*)^{2-\dim(\sigma_v) - 1} \times \prod_{e \in E(\Gamma^c)} (\mathbb{C}^*)^{\dim(\sigma_e)-2}\]

    The exponent of \(\mathbb{C}^*\) can be written as follows: 
    
    \begin{align*}
        &=\sum_{v \in V(\Gamma^c)} (2 - \dim(\sigma_v)) + \sum_{v \in V(\Gamma^c) \setminus V(\Gamma)} (-1) + \sum_{e \in E(\Gamma^c)} (\dim(\sigma_e)-2) \\  
        &= {2(|V(\Gamma^c)| - |E(\Gamma^c)|)}  + \sum_{e \in E(\Gamma^c)}{\dim(\sigma_e)} -  \sum_{v \in V(\Gamma^c)}{\dim(\sigma_v)} - \sum_{v \in V(\Gamma^c) \setminus V(\Gamma)} {1}
    \end{align*}\[\] Then, using that \(|V(\Gamma^c)| - |E(\Gamma^c)| = 1\), we see that this simplifies to \[(\mathbb{C}^*)^{\sum_{e \in E(\Gamma^c)} \dim(\sigma_e) - \sum_{v \in V(\Gamma^c)} \dim(\sigma_v)+2 - |V(\Gamma^c) \setminus V(\Gamma)|} = (\mathbb{C}^*)^{R_c}.\] Combining with \(\prod_{v \in V(\Gamma)}\mathcal{M}_{0, \text{val}(v)}\) gives \eqref{form}. 
\end{proof}

    \subsection{\(\Sigma\)-Equivalence}\label{equiv}

We now define a notion of equivalence of tangency data such that two sets of data \(\upalpha\), \(\upalpha'\) are equivalent if and only if they are in the same chamber of the \(\Sigma\)-slope decomposition. Throughout the rest of this chapter, for all \(I \subset [n]\), assume that \(\upalpha_I \neq \colVec{0}{0}\). We will discuss the alternative case in Remark \ref{zero}. 

We begin with a subdivision of \(\Sigma\). 

\begin{definition}[\(\Sigma^\dagger\)]
    Let \(\Sigma\) be a 2-dimensional fan and \(\Sigma(1) = \{\rho_1, \ldots, \rho_m\}\) where \(\rho_i\) has slope \(v_{\rho_i} \in N\). Define \(-\rho_i := \cone(-v_{\rho_i})\). Then, \(\Sigma^\dagger\) is defined to be the complete 2-dimensional fan such that \({\Sigma^{\dagger}}(1) = \{\pm \rho_1, \ldots, \pm \rho_m\}\).
\end{definition}

\begin{example}[Defining \(\Sigma^\dagger\)]
    Taking the fan of \(\mathbb{P}^2\), to form \(\Sigma^\dagger\), we include the rays \(\colVec{-1}{0}\), \(\colVec{0}{-1}\) and \(\colVec{1}{1}\):
    
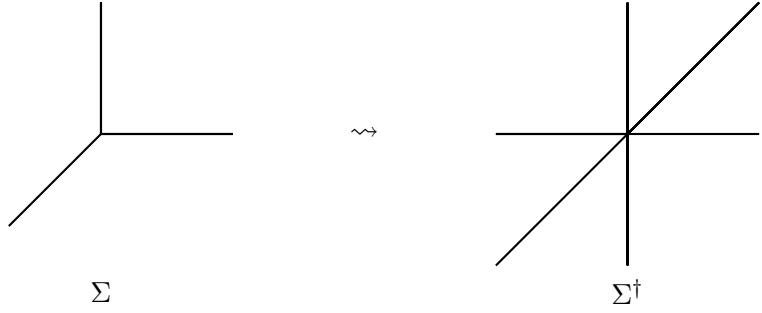
\begin{center}
        \begin{tikzpicture}[scale = 1.75]
            \draw[thick] (0,0) -- (1,0) -- (0,0) -- (0,1) -- (0,0) -- (-0.7, -0.7);
            \node at (2, 0) {\(\rightsquigarrow\)};
            \draw[thick] (5, 0) -- (3, 0) -- (4,0) -- (4, 1) -- (4, -1) -- (4,0) --  (5, 1) -- (3, -1);
            \node at (0, -1.2) {\(\Sigma\)};
            \node at (4, -1.2) {\(\Sigma^\dagger\)};
        \end{tikzpicture}
    \end{center}

    It could be the case, as for \(\mathbb{P}^1 \times \mathbb{P}^1\), that \(\Sigma = \Sigma^\dagger\). \begin{flushright}\emph{\textbf{End of example.}}\end{flushright}
\end{example}

\begin{definition}[Cyclic Pre-order, \cite{cycle}] \label{cyclicOrdering}
    A \emph{cyclic pre-order} on a set \(X\) is a collection of triples \([a,b,c] \in \mathscr{C} \subset X^3\) satisfying the following axioms:
    
    If \(a\), \(b\), \(c \in X\) are distinct, then 
    \begin{enumerate}
        \item If \([a, b, c] \in \mathscr{C}\), then \([b, c, a] \in \mathscr{C}\).
        \item Either \([a,b,c] \in \mathscr{C}\) or \([c,b,a] \in \mathscr{C}\). 
        \item If \([a, b, c] \in \mathscr{C}\), then \([c, b, a] \not \in \mathscr{C}\). 
        \item If \([a, b, d] \in \mathscr{C}\) and \([b, c, d] \in \mathscr{C}\), then \([a, b, c] \in \mathscr{C}\). 
    \end{enumerate} And finally,
    \begin{enumerate}[start = 5]
        \item If \([a,b,c] \in \mathscr{C}\) then \(a\), \(b\) and \(c\) are distinct. 
    \end{enumerate}
\end{definition}

Now, define \[\mathcal{I}^{\Sigma} = \left(\mathcal{P}([n]) \setminus \{\emptyset, [n]\}\right) \cup \Sigma^\dagger(1).\] Elements of \(\mathcal{I}^\Sigma\) define half-lines through the origin of slope \(\upalpha_I\) for \(I \subset [n]\) and of slope \(m_{\rho}\) for \(\rho \in \Sigma^\dagger(1)\). Given \(\upalpha\), for \(p \in \mathcal{I}^\Sigma\), write \(m_{\upalpha}(p) = \colVec{p_x}{p_y} \in N\) for the corresponding slope. There is a natural map \begin{align*} \nu_{\upalpha}\colon \mathcal{I}^\Sigma &\to S^1\\
    p &\mapsto \frac{1}{\sqrt{p_x^2+p_y^2}} \colVec{p_x}{p_y}
\end{align*} sending a vector to its normalisation on \(S^1\). As a subset of \(S^1\), we can define a cyclic ordering on the image \(\nu_{\upalpha}(\mathcal{I}^\Sigma)\) as per Definition \ref{cyclicOrdering}.

\begin{definition}[Cyclic Ordering on \(\mathcal{I}^\Sigma\)]\label{cyclicOrderDef}
    Given \(\upalpha \in N_0^n\) and \(\mathcal{I}^\Sigma\) for a complete fan \(\Sigma\), define a cyclic ordering on \(\mathcal{I}^\Sigma\) by using the cyclic ordering on the circle: for \(p\), \(q\) and \(r \in \mathcal{I}^\Sigma\), write \([p,q,r]\) if and only if \([\nu_{\upalpha}(p), \nu_{\upalpha}(q), \nu_{\upalpha}(r)]\) on \(S^1\). We denote this by writing \([p,q,r] \in \mathscr{C}_\upalpha(\mathcal{I}^\Sigma)\) where \(\mathscr{C}_\upalpha(\mathcal{I}^\Sigma)\) is the set of all cyclically ordered triples. 
\end{definition}

\begin{example}
    Let \(\Sigma\) be the fan of \(\mathbb{P}^2\) and \(\upalpha = \left(\colVec{1}{2}, \colVec{1}{3}, \colVec{-2}{-5}\right)\). Then, projecting each ray of slope \(\upalpha_I\) and \(v_\rho\) onto \(S^1\) gives 

    
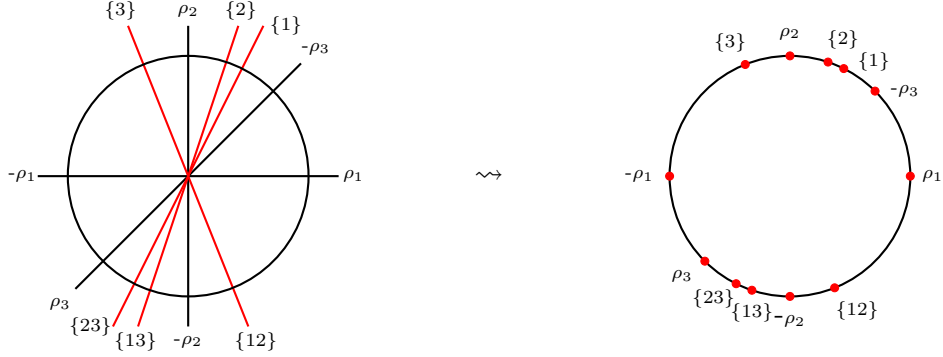
\begin{center}
        \begin{tikzpicture}[scale = 2]
            \draw[thick] (-1, 0) -- (1,0);
            \draw[thick] (0,-1) -- (0,1);
            \draw[thick] (-0.75, -0.75) -- (0.75,0.75);
            \draw[red, thick] (0.5, 1) -- (-0.5, -1);
            \draw[red, thick] (0.33333,1) -- (-0.33333, -1);
            \draw[red, thick] (0.4,-1) -- (-0.4, 1);
            \node at (1.1, 0) {\tiny{\(\rho_1\)}};
            \node at (-1.1, 0) {\tiny{-\(\rho_1\)}};
            \node at (0.85, 0.85) {\tiny{-\(\rho_3\)}};
            \node at (-0.85, -0.85) {\tiny{\(\rho_3\)}};
            \node at (0.65, 1) {\tiny{\(\{1\}\)}};
            \node at (-0.65, -1) {\tiny{\(\{23\}\)}};
            \node at (0.35, 1.1) {\tiny{\(\{2\}\)}};
            \node at (-0.35, -1.1) {\tiny{\(\{13\}\)}};
            \node at (0, 1.1) {\tiny{\(\rho_2\)}};
            \node at (0, -1.1) {\tiny{-\(\rho_2\)}};
            \node at (-0.45, 1.1) {\tiny{\(\{3\}\)}};
            \node at (0.45, -1.1) {\tiny{\(\{12\}\)}};

            \node at (2, 0) {\(\rightsquigarrow\)} ;
            \draw[thick] (0,0) circle (0.8cm);
            \draw[thick] (4,0) circle (0.8cm);
            \filldraw[red] (0+4,1*0.8) circle (0.75pt);
            \filldraw[red] (0+4,-1*0.8) circle (0.75pt);
            \filldraw[red] (1*0.8+4,0) circle (0.75pt);
            \filldraw[red] (-1*0.8+4,0) circle (0.75pt);
            \filldraw[red] (0.8*0.31623+4, 0.8*0.94868) circle (0.75pt);
            \filldraw[red] (-0.31623*0.8+4, -0.94868*0.8) circle (0.75pt);
            \filldraw[red] (0.44721*0.8+4, 0.89443*0.8) circle (0.75pt);
            \filldraw[red] (-0.44721*0.8+4, -0.89443*0.8) circle (0.75pt);
            \filldraw[red] (-0.37139*0.8+4, 0.92848*0.8) circle (0.75pt);
            \filldraw[red] (0.37139*0.8+4, -0.92848*0.8) circle (0.75pt);
            \filldraw[red] (0.70711*0.8 +4, 0.70711*0.8) circle (0.75pt);
            \filldraw[red] (-0.70711*0.8 +4, -0.70711*0.8) circle (0.75pt);

            \node at (1*0.8+4 + 0.15,0) {\tiny{\(\rho_1\)}}; 
            \node at (0+4,1*0.8+0.15) {\tiny{\(\rho_2\)}};
            \node at (0+4,-1*0.8-0.15) {-\tiny{\(\rho_2\)}};
            \node at (-1*0.8+4-0.2,0) {\tiny{-\(\rho_1\)}};
            \node at (0.8*0.31623+4+0.1, 0.8*0.94868+0.15) {\tiny{\(\{2\}\)}};
            \node at (-0.31623*0.8+4, -0.94868*0.8-0.15) {\tiny{\(\{13\}\)}};
            \node at (0.44721*0.8+4+0.2, 0.89443*0.8+0.05) {\tiny{\(\{1\}\)}};
            \node at (-0.44721*0.8+4-0.15, -0.89443*0.8-0.1) {\tiny{\(\{23\}\)}};
            \node at (-0.37139*0.8+4-0.1, 0.92848*0.8+0.15) {\tiny{\(\{3\}\)}};
            \node at (0.37139*0.8+4+0.15, -0.92848*0.8-0.15) {\tiny{\(\{12\}\)}};
            \node at (0.70711*0.8 +4+0.2, 0.70711*0.8) {\tiny{-\(\rho_3\)}};
            \node at (-0.70711*0.8 +4-0.15, -0.70711*0.8-0.1) {\tiny{\(\rho_3\)}};
        \end{tikzpicture}
    \end{center} We have, for example, \([-\rho_3, \{1\}, \{2\}]\) and \([-\rho_1, \{12\}, \rho_1] \in \mathscr{C}_{\upalpha}(\mathcal{I}^\Sigma)\).
    \begin{flushright} \textbf{\emph{End of example.}} \end{flushright}
\end{example}

We use the cyclic ordering on \(\mathcal{I}^\Sigma\) to define an equivalence of tangency data. 

\begin{definition}[\(\Sigma\)-Equivalence]\label{sigmaEquiv}
    Let \(\upalpha\), \(\upalpha' \in N_0^n\). We say \(\upalpha\) and \(\upalpha'\) are \emph{\(\Sigma\)-equivalent}, written \(\upalpha \sim_\Sigma \upalpha'\), if the cyclic orderings on \(\mathcal{I}^\Sigma\) induced by \(\nu_{\upalpha}\) and \(\nu_{\upalpha'}\) are the same. 
\end{definition} \noindent We will define the walls of the \(\Sigma\)-slope decomposition to define regions where the cyclic orderings are constant in Section \ref{walls}. The open regions between walls will be referred to as \emph{chambers}.

\begin{remark}\label{nu}
    Definition \ref{cyclicOrderDef} captures when, for \(p\), \(q \in \mathcal{I}^\Sigma\), \(\nu_\upalpha(p) = \nu_\upalpha(q)\). In this instance, we have that for all \(r\), \(s \in \mathcal{I}^\Sigma\), \[[p, r,s] \in \mathscr{C}_{\upalpha}(\mathcal{I}^\Sigma) \iff [q, r, s] \in \mathscr{C}_{\upalpha}(\mathcal{I}^\Sigma)\] and \[\nexists \,z \in \mathcal{I}^\Sigma : [p, q, z] \in \mathscr{C}_{\upalpha}(\mathcal{I}^\Sigma).\] Therefore, for \(\upalpha\), \(\upalpha'\) to be \(\Sigma\)-equivalent, for all \(p\), \(q \in \mathcal{I}^\Sigma\), we must have \[\nu_\upalpha(p) = \nu_{\upalpha}(q) \iff \nu_{\upalpha'}(p) = \nu_{\upalpha'}(q).\] \begin{flushright} \textbf{\emph{End of remark.}} \end{flushright}
\end{remark}

\begin{definition}[Conical Containments, Combinatorially Equivalent]
    Let \(\mathcal{A} = \{a_1, \ldots, a_m\}\) be an ordered set of vectors in \(N_\R\). Then, the \emph{conical containments} of \(\mathcal{A}\) are the containments of the form \[a_i \in \cone(a_j), \quad a_i \in \cone(a_k, a_l)\] that \(\mathcal{A}\) satisfies. 

    Two ordered sets of vectors \(\mathcal{A} = \{a_1, \ldots, a_m\}\), \(\mathcal{B} = \{b_1, \ldots, b_m\}\) are called \emph{combinatorially equivalent} if they have the same cone containments, i.e. \[a_i \in \cone(a_j) \iff b_i \in \cone(b_j) \] and \[a_i \in \cone(a_k, a_l) \iff b_i \in \cone(b_k, b_l).\]
\end{definition}

\begin{lemma}
    Let \(\upalpha\), \(\upalpha' \in N_0^n\). If \(\upalpha \sim_\Sigma \upalpha'\), then \(\nu_{\upalpha}(\mathcal{I}^\Sigma)\), \(\nu_{\upalpha'}(\mathcal{I}^\Sigma)\) are combinatorially equivalent. \label{coneConditions}
\end{lemma}

\begin{proof}
    We begin with one-dimensional conical containments. For \(\upalpha\), let \(p\), \(q \in \mathcal{I}^\Sigma\) such that \(m_{\upalpha}(p) \in \cone(m_{\upalpha}(q))\). We have \[m_{\upalpha}(p) \in \cone(m_{\upalpha}(q)) \iff \nu_{\upalpha}(p) = \nu_{\upalpha}(q)\] This holds for \(\upalpha\) if and only if it holds for \(\upalpha'\) by Remark \ref{nu}. 

    Then, for maximal cones, consider \(p\), \(q\), \(r \in \mathcal{I}^\Sigma\), all with distinct images under \(\nu_{\upalpha}\) (and so, under \(\nu_{\upalpha'}\)). Suppose that \(m_{\upalpha}(p) \in \cone(m_{\upalpha}(q), m_{\upalpha}(r))\), i.e. we have (without loss of generality) \([q, p,r] \in \mathscr{C}_{\upalpha}(\mathcal{I}^\Sigma)\). Thus, \(m_{\upalpha}(p)\), \(m_{\upalpha}(q)\) and \(m_{\upalpha}(r)\) are contained within a half space. In particular, they are contained within the half-space determined by \(m_{\upalpha}(q)\) and \(-m_{\upalpha}(q)\). 
    The picture is as follows: 

    
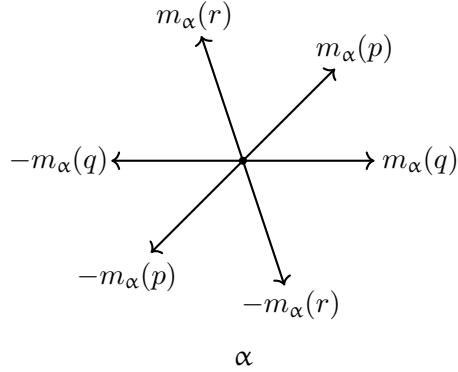
\begin{center}
        \begin{tikzpicture}[scale = 1.75]
            \filldraw (0,0) circle (0.75pt);
            \draw[thick, <->] (-1, 0) -- (1, 0);
            \draw[thick, <->] (-0.7, -0.7) -- (0.7, 0.7); 
            \draw[thick, <->] (-0.31623, 0.94868) -- (0.31623, -0.94868); 
            \node at (1.35, 0) {\(m_{\upalpha}(q)\)};
            \node at (-1.4, 0.) {\(-m_{\upalpha}(q)\)};
            \node at (0.85, 0.85) {\(m_{\upalpha}(p)\)};
            \node at (-0.88, -0.88) {\(-m_{\upalpha}(p)\)};
            \node at (-0.31623-0.05, 0.94868+0.15) {\(m_{\upalpha}(r)\)};
            \node at (0.31623+0.05, -0.94868-0.15) {\(-m_{\upalpha}(r)\)};
            \node at (0, -1.5) {\(\upalpha\)};
        \end{tikzpicture}
    \end{center}
    
    Now, consider \(m_{\upalpha'}(p)\), \(m_{\upalpha'}(q)\) and \(m_{\upalpha'}(r)\). Suppose that \(m_{\upalpha'}(p) \not \in \cone(m_{\upalpha'}(q), m_{\upalpha'}(r))\). Then, either they are still contained within a half-space, but have swapped orders, or they are no longer contained within a half space. Assume firstly that \(m_{\upalpha'}(p)\), \(m_{\upalpha'}(q)\) and \(m_{\upalpha'}(r)\) are contained within some half space. There are two possibilities: 
    
    
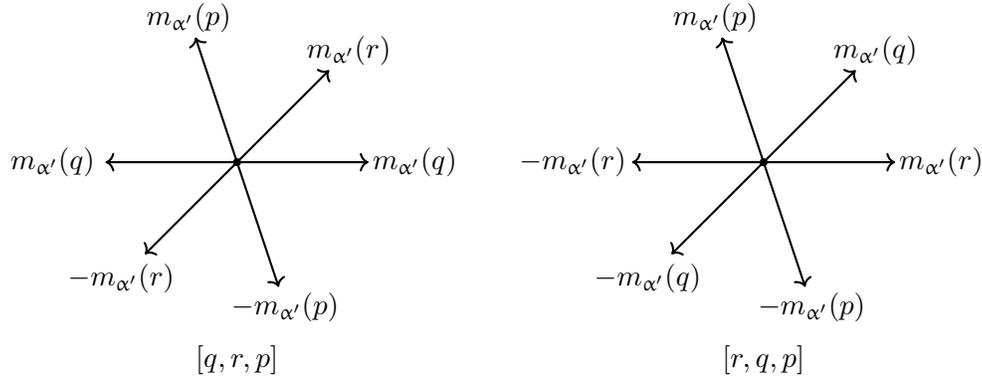
\begin{center}
    \begin{tikzpicture}[scale = 1.75]
        \filldraw (0,0) circle (0.75pt);
        \draw[thick, <->] (-1, 0) -- (1, 0);
        \draw[thick, <->] (-0.7, -0.7) -- (0.7, 0.7); 
        \draw[thick, <->] (-0.31623, 0.94868) -- (0.31623, -0.94868); 
        \node at (1.35, 0) {\(m_{\upalpha'}(r)\)};
        \node at (-1.45, 0.) {\(-m_{\upalpha'}(r)\)};
        \node at (0.85, 0.85) {\(m_{\upalpha'}(q)\)};
        \node at (-0.88, -0.88) {\(-m_{\upalpha'}(q)\)};
        \node at (-0.31623-0.05, 0.94868+0.15) {\(m_{\upalpha'}(p)\)};
        \node at (0.31623+0.05, -0.94868-0.15) {\(-m_{\upalpha'}(p)\)};
        \node at (0, -1.5) {\([r,q,p]\)};

        \filldraw (0-4,0) circle (0.75pt);
        \draw[thick, <->] (-1-4, 0) -- (1-4, 0);
        \draw[thick, <->] (-0.7-4, -0.7) -- (0.7-4, 0.7); 
        \draw[thick, <->] (-0.31623-4, 0.94868) -- (0.31623-4, -0.94868); 
        \node at (1.35-4, 0) {\(m_{\upalpha'}(q)\)};
        \node at (-1.4-4, 0.) {\(m_{\upalpha'}(q)\)};
        \node at (0.85-4, 0.85) {\(m_{\upalpha'}(r)\)};
        \node at (-0.88-4, -0.88) {\(-m_{\upalpha'}(r)\)};
        \node at (-0.31623-0.05-4, 0.94868+0.15) {\(m_{\upalpha'}(p)\)};
        \node at (0.31623+0.05-4, -0.94868-0.15) {\(-m_{\upalpha'}(p)\)};
        \node at (0-4, -1.5) {\([q,r,p]\)};
    \end{tikzpicture}
\end{center}

    Thus, we must have either \([q,r,p] \in \mathscr{C}_{\upalpha'}(\mathcal{I}^\Sigma)\) or \([r,q,p] \in \mathscr{C}_{\upalpha'}(\mathcal{I}^\Sigma)\). The presence of \([q,r,p]\) is an immediate contradiction to \(\upalpha \sim_\Sigma \upalpha'\) by Definition \ref{cyclicOrdering} since \([q,p,r] \in \mathscr{C}_{\upalpha}(\mathcal{I}^\Sigma)\). Denote by \(-q \in \mathcal{I}^\Sigma\) the element corresponding either to, if \(p \subset [n]\), \(p^{\mathrm{c}})\), or if \(q \in \Sigma(1)\), the ray of opposite slope. Then, \(m_{\upalpha}(-q) = -m_{\upalpha}(q)\). Now, if \([r,q,p] \in \mathscr{C}_{\upalpha'}(\mathcal{I}^\Sigma)\) we then have \([r,p,-q] \in \mathscr{C}_{\upalpha'}(\mathcal{I}^\Sigma)\), yet \([p,r,-q] \in \mathscr{C}_{\upalpha}(\mathcal{I}^\Sigma)\), giving a contradiction. 
    
    Finally, assume that \(m_{\upalpha'}(p) \not \in \cone(m_{\upalpha'}(q), m_{\upalpha'}(r))\) because\(m_{\upalpha'}(p)\), \(m_{\upalpha'}(q)\) and \(m_{\upalpha'}(r)\) are not all contained in a half-space. Consider the half-space \(m_{\upalpha'}(q)- (-m_{\upalpha'}(q))\): 
    
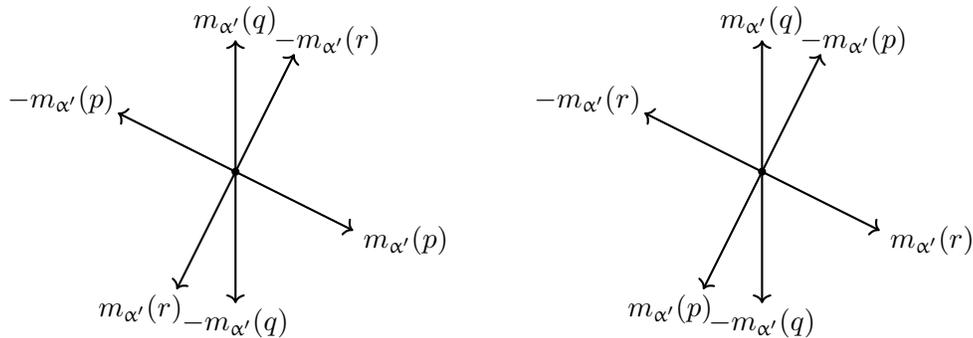
    \begin{center}
        \begin{tikzpicture}[scale = 1.75]
            \draw[thick, <->] (0, 1) -- (0,-1);
            \draw[thick, <->] (-0.44721, -0.89443) -- (0.44721, 0.89443);
            \draw[thick, <->] (0.89443, -0.44721) -- (-0.89443, 0.44721);
            \filldraw (0,0) circle (0.75pt);
            \node at (0, 1.15) {\(m_{\upalpha'}(q)\)};
            \node at (0, -1.15) {\(-m_{\upalpha'}(q)\)};
            \node at (-0.44721-0.075-0.2, -0.89443-0.15) {\(m_{\upalpha'}(r)\)};
            \node at (0.44721+0.05+0.2, 0.89443+0.1) {\(-m_{\upalpha'}(r)\)};
            \node at (-0.89443-0.43, 0.44721+0.1) {\(-m_{\upalpha'}(p)\)};
            \node at (0.89443+0.4, -0.44731-0.075) {\(m_{\upalpha'}(p)\)};

            \draw[thick, <->] (0+4, 1) -- (0+4,-1);
            \draw[thick, <->] (-0.44721+4, -0.89443) -- (0.44721+4, 0.89443);
            \draw[thick, <->] (0.89443+4, -0.44721) -- (-0.89443+4, 0.44721);
            \filldraw (0+4,0) circle (0.75pt);
            \node at (0+4, 1.15) {\(m_{\upalpha'}(q)\)};
            \node at (0+4, -1.15) {\(-m_{\upalpha'}(q)\)};
            \node at (-0.44721-0.05+4-0.2, -0.89443-0.15) {\(m_{\upalpha'}(p)\)};
            \node at (0.44721+0.05+4+0.2, 0.89443+0.1) {\(-m_{\upalpha'}(p)\)};
            \node at (-0.89443-0.43+4, 0.44721+0.1) {\(-m_{\upalpha'}(r)\)};
            \node at (0.89443+0.4+4, -0.44731-0.075) {\(m_{\upalpha'}(r)\)};
        \end{tikzpicture}
    \end{center} We either have \([q, r, -q] \in \mathscr{C}_{\upalpha'}(\mathcal{I}^\Sigma)\) or \([q, p, -q] \in \mathscr{C}_{\upalpha'}(\mathcal{I}^\Sigma)\). If \([q, r, -q] \in \mathscr{C}_{\upalpha'}(\mathcal{I}^\Sigma)\), we have \([q, r, p] \in \mathscr{C}_{\upalpha'}(\mathcal{I}^\Sigma)\), contradicting \([q, p,r] \in \mathscr{C}_{\upalpha}(\mathcal{I}^\Sigma)\). Else, if \([q, p, -q] \in \mathscr{C}_{\upalpha'}(\mathcal{I}^\Sigma)\), we have \([-q, r, q] \in \mathscr{C}_{\upalpha'}(\mathcal{I}^\Sigma)\), a contradiction to \([-q, q, r] \in \mathscr{C}_{\upalpha}(\mathcal{I}^\Sigma)\). Therefore, \(m_{\upalpha'}(p) \in \cone(m_{\upalpha'}(q), m_{\upalpha'}(r))\) and \(\nu_{\upalpha}(\mathcal{I}^\Sigma)\), \(\nu_{\upalpha'}(\mathcal{I}^\Sigma)\) are combinatorially equivalent.
\end{proof}

We are now in a position to state the main theorem:

\begin{theorem}[Theorem \ref{thm: intro}] \label{thm: body} As a function of \(\upalpha\), the class \(\left[\,\overbar{\mathcal{M}}(X_\Sigma, \upalpha)\right] \in K_0(\mathrm{Var}/\mathbb{C})\) is constant on the chambers of the \(\Sigma\)-slope decomposition of \(N^n\).     
\end{theorem}

    \subsection{Proof part 1: the slope-sensitive fan}\label{fan}

The following construction uses ideas of \cite[Section 2.3]{BNR2022}. We define a further subdivision of \(\Sigma\), denoted \(\tilde{\Sigma}_{\upalpha}\), depending on \(\upalpha\). This refinement is needed for the proof of Theorem \ref{thm: body}, however, it is not required for the statement. 

The subdivision \(\tilde{\Sigma}_{\upalpha}\) is referred to as the \emph{slope sensitive fan} given \(\upalpha\). To define it, we begin with \(\Sigma^\dagger\). 

\begin{definition}[Slope-Sensitive Fan]\label{ssFan}
    Consider a complete 2-dimensional fan \(\Sigma\) in  \(N_\R\) and \(\upalpha \in N^n_0\). For each \(I \subset [n]\), define \(\rho_I := \cone(\upalpha_I)\). We define the \emph{slope-sensitive fan}, \(\tilde{\Sigma}_{\upalpha}\), to be the complete 2-dimensional fan such that \[\tilde{\Sigma}_{\upalpha}(1) = {\Sigma^{\dagger}}(1) \cup \{\rho_I : I \subset [n]\}.\] 
\end{definition}

By the balancing condition, for every ray \(\rho_I\) with slope \(\upalpha_I\), we have \(\rho_{I} = - \rho_{I^\mathrm{c}}\) and so for all rays of the slope-sensitive fan, there will also be a ray of the opposite slope. Note that \(\tilde{\Sigma}_{\upalpha}\) may be not be smooth, but this is unimportant for the proof of Theorem \ref{thm: body}.  

\begin{example}[Slope Sensitive Fan]
    \label{slopeSen} Let \(\upalpha = \left(\colVec{2}{0}, \colVec{0}{1}\right), \colVec{-2}{-1})\) and \(\Sigma = \Sigma_{\mathbb{P}^2}\). We first take \(\Sigma_{\mathbb{P}^2}^\dagger\). Then, to arrive at the slope sensitive fan, we add rays of slope \[\quad \upalpha_1 = \colVec{2}{0}, \quad \upalpha_2 = \colVec{0}{1}, \quad \upalpha_3 = \colVec{-2}{-1}\] along with \[\upalpha_1 + \upalpha_2 = -\upalpha_3 = \colVec{2}{1}, \quad \upalpha_2 + \upalpha_3 = - \upalpha_1 = \colVec{-2}{0}, \quad \upalpha_1 + \upalpha_3 = - \upalpha_2 = \colVec{0}{-1}.\] Some rays are already present in \(\Sigma^\dagger_{\mathbb{P}^2}\), for example \(\colVec{2}{0}\) since it is parallel to \(\colVec{1}{0}\). The slope sensitive fan then looks as follows, with new rays shown in orange:

    
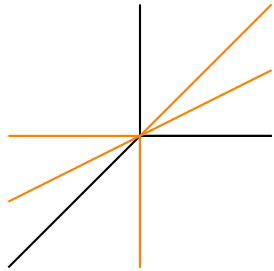
\begin{center}
        \begin{tikzpicture}[scale = 1.75] 
            \draw[thick] (0,1) -- (0,0) -- (1,0) -- (0,0) -- (-1,-1);
            \draw[orange, thick] (-1, 0) -- (0,0) -- (0, -1) -- (0,0) -- (1,1);
            \draw[orange, thick] (1, 0.5) -- (-1, -0.5);
        \end{tikzpicture}
    \end{center} \begin{flushright}\emph{\textbf{End of example.}}\end{flushright}   
\end{example}

\begin{definition}[Analogous Cone]
    Let \(\upalpha\), \(\upalpha' \in N^n_0\) and fan \(\Sigma\) with \(\Sigma^\dagger(1) = \{\rho_1, \ldots, \rho_k\}\). Let \(\tilde{\Sigma}_{\upalpha}\), \(\tilde{\Sigma}_{\upalpha'}\) be the corresponding slope-sensitive fans and let \(\sigma \in \tilde{\Sigma}_{\upalpha}\) be a cone. The \emph{analogous cone} to \(\sigma\), denoted \(\sigma'\), is defined in the following way: \[\sigma = \cone(\upalpha_{I_1}, \ldots, \upalpha_{I_m}, \rho_{i_1}, \ldots, \rho_{i_\ell}) \implies \sigma' = \cone(\upalpha'_{I_1}, \ldots, \upalpha'_{Im}, \rho_{i_1}, \ldots, \rho_{i_\ell})\] where \(I_1, \ldots, I_m \subset [n]\). 
\end{definition}

We now give a lemma.

\begin{lemma}
    Let \(\upalpha \sim_\Sigma \upalpha'\). Then, for each cone \(\sigma \in \tilde{\Sigma}_{\upalpha}\), the analogous cone \(\sigma'\) is a cone of \(\tilde{\Sigma}_{\upalpha'}\). \label{analogousCone}
\end{lemma}

\begin{proof}
    For one-dimensional cones, all rays of \(\Sigma^\dagger\) and those of the form \(\cone(\upalpha_I)\) are cones of both \(\tilde{\Sigma}_{\upalpha}\), \(\tilde{\Sigma}_{\upalpha'}\) by construction. Let \(p\), \(q \in \mathcal{I}^\Sigma\) such that \(\sigma = \cone(m_{\upalpha}(p), m_{\upalpha}(q))\) is a maximal cone of \(\tilde{\Sigma}_{\upalpha}\). Then, for all \(r \in \mathcal{I}^\Sigma\) such that \(\nu_\upalpha(r) \neq \nu_\upalpha(p), \nu_\upalpha(q)\), we have, without loss of generality, \([p,q, r] \in \mathscr{C}_{\upalpha}(\mathcal{I}^\Sigma)\), and by Definition \ref{cyclicOrdering}, there exists no \(s \in \mathcal{I}^\Sigma\) such that \([p, s, q] \in \mathscr{C}_{\upalpha}(\mathcal{I}^\Sigma)\). 
    
    Consider the analogous cone \(\sigma' = \cone(m_{\upalpha'}(p), m_{\upalpha'}(q))\). Since \(\upalpha \sim_\Sigma \upalpha'\), we have \(\nu_{\upalpha'}(p) \neq \nu_{\upalpha'}(q)\) and hence \(\sigma'\) is 2-dimensional region. Assume that \(\sigma'\) is not a cone of \(\tilde{\Sigma}_{\upalpha'}\). Then, there exists a ray \(\rho \in \tilde{\Sigma}_{\upalpha'}(1)\) strictly between \(\cone(m_{\upalpha'}(p))\) and \(\cone(m_{\upalpha'}(q))\) corresponding to some \(s \in \mathcal{I}^\Sigma\). Therefore, we have \([p, s,q] \in \mathscr{C}_{\upalpha'}(\mathcal{I}^\Sigma)\) with \(\nu_{\upalpha'}(s) \neq \nu_{\upalpha'}(p)\), \(\nu_{\upalpha'}(q)\). But, no such triple exists in \(\mathscr{C}_{\upalpha}(\mathcal{I}^\Sigma)\), contradicting that \(\upalpha \sim_\Sigma \upalpha'\).
\end{proof}

\subsection{Proof part 2: tropical lifting} \label{lift}
We will use \(\tilde{\Sigma}_{\upalpha}\) as ``scaffolding'' in the proof of Theorem \ref{thm: body} to build the bijection of combinatorial types \(c \leftrightarrow c'\).

\begin{definition}
    Let \(\pi\colon \tilde{\Sigma} \to \Sigma\) be a subdivision. For a cone \(\sigma \in \tilde{\Sigma}\), define \(\pi_*(\sigma)\) to be the smallest cone containing the image of \(\sigma\) under \(\pi\). 

\end{definition}

Given the subdivision \(\pi\colon \tilde{\Sigma}_{\upalpha} \to \Sigma\), we have the following definition. 

\begin{definition}[Lift of a Combinatorial Type]
     Fix \(\Gamma \in \Gamma_{0,n}\), \(\upalpha \in N_0^n\) and a subdivision \(\pi\colon \tilde{\Sigma} \to \Sigma\). Consider \(\mathcal{C}(\Gamma, \Sigma, \upalpha)\), \(\mathcal{C}(\Gamma, \tilde{\Sigma}, \upalpha)\). There is an induced map \[f\colon \mathcal{C}(\Gamma, \tilde{\Sigma}, \upalpha) \to \mathcal{C}(\Gamma, \Sigma, \upalpha)\] defined in the following way. Let \(\tilde{\Sigma}(1) \setminus \Sigma(1) = \{\rho_1, \ldots, \rho_k\} \). Given a combinatorial type \(\tilde{c} \in \mathcal{C}(\Gamma, \tilde{\Sigma}, \upalpha)\) with semi-stable graph \(\Gamma^{\tilde{c}}\) and cone assignments \(v \mapsto \tilde{\sigma}_v\), the image \(f(\tilde{c}) = c\) where \(c\) has the following data:

    \begin{enumerate}
        \item Semi-stable graph \(\Gamma^c\) which is a partial stabilisation of \(\Gamma^{\tilde{c}}\): \[V(\Gamma^c) = V(\Gamma^{\tilde{c}}) \setminus \{v \in V(\Gamma^{\tilde{c}})\mid v \text{ 2-valent and }\tilde{\sigma}_v = \rho_i \text { for some } i\}.\] The edge set \(E(\Gamma^c)\) is obtained from \(E(\Gamma^{\tilde{c}})\) after smoothing
        semi-stable vertices.
        \item For \(v \in V(\Gamma^c)\), \(e \in E(\Gamma^c)\), \(\ell \in L(\Gamma^c)\), cone assignments \(v \mapsto \pi_*(\tilde{\sigma}_v)\), \(e \mapsto \pi_*(\tilde{\sigma}_e)\) and \(\ell \mapsto \pi_*(\tilde{\sigma}_\ell)\). 
        \item Balanced slope assignment \(m_{\upalpha}: \vv{E}(\Gamma^c) \cup \vv{L}(\Gamma^c)\) induced from the assignment for \(\Gamma^{\tilde{c}}\). 
    \end{enumerate} A \emph{lift} of \(c \in \mathcal{C}(\Gamma, \Sigma, \upalpha)\) is a choice of  \(\tilde{c} \in f^{-1}(c)\).  
\end{definition}

\begin{proposition}
    Let \(\upalpha \in N_0^n\), \(\Gamma \in \Gamma_{0,n}\) and consider a subdivision of fans \(\pi\colon \tilde{\Sigma} \to \Sigma\). Then, for each realisable combinatorial type \(c \in \mathcal{C}(\Gamma, \Sigma, \upalpha)\), there exists a lift \(\tilde{c} \in \mathcal{C}(\Gamma, \tilde{\Sigma}, \upalpha)\) which is realisable. \label{realLift}
\end{proposition}

\begin{proof}
    Since \(c\) is realisable, there exists a tropical map \(f\colon \sqC_\Gamma \to \Sigma\) such that the combinatorial type of \(f\), as per Construction \ref{construction}, is \(c\). Given \(f\), subdivide \(\Sigma\) to obtain \(\tilde{\Sigma}\). This induces a tropical map \(\tilde{f}\colon \sqC_\Gamma \to \tilde{\Sigma}\). Taking the combinatorial type of \(\tilde{f}\) gives a realisable combinatorial type \(\tilde{c}\) lifting \(c\). 
\end{proof}


We give an example. 

\begin{example}[Lifting a Combinatorial Type]
    Let \(X_\Sigma = \mathbb{P}^2\) and \(\upalpha = \left(\colVec{2}{0}, \colVec{0}{1}, \colVec{-2}{-1}\right)\). Let \((C, p_1, p_2, p_3)\) be smooth. Consider the combinatorial types \(c_1\), \(c_2\) and \(c_3\):
    \begin{center}
        \begin{tikzpicture}[scale = 1.75]
            \draw[thick] (0,0) -- (0,1) -- (0,0) -- (1,0) -- (0,0) -- (-0.7, -0.7);
            \filldraw[red] (.5, 0.25+0.2) circle (1pt);
            \draw[red, thick, ->] (0.5, 0.25+0.2) -- (-0.7, -0.35+0.2);
            \draw[red, thick, ->] (0.5, 0.25+0.2) -- (1, 0.25+0.2);
            \draw[red, thick, ->] (0.5, 0.25+0.2) -- (0.5, 1);
            \filldraw[white] (0,0.2) circle (1pt);
            \draw[red] (0, 0.2) circle (1pt);
            \node at (0, -1) {\(c_1\)};

            \draw[thick] (0+3,0) -- (0+3,1) -- (0+3,0) -- (1+3,0) -- (0+3,0) -- (-0.7+3, -0.7);
            \filldraw[red] (.5+3, 0.25) circle (1pt);
            \draw[red, thick, ->] (0.5+3, 0.25) -- (-0.7+3, -0.35);
            \draw[red, thick, ->] (0.5+3, 0.25) -- (1+3, 0.25);
            \draw[red, thick, ->] (0.5+3, 0.25) -- (0.5+3, 1);
            \filldraw[white] (3,0) circle (1pt);
            \draw[red] (3, 0) circle (1pt);
            \node at (3, -1) {\(c_2\)};
            
            \draw[thick] (0+6,0) -- (0+6,1) -- (0+6,0) -- (1+6,0) -- (0+6,0) -- (-0.7+6, -0.7);
            \draw[red, thick, ->] (0.5+6, 0.25-0.1) -- (-0.7+6, -0.35-0.1);
            \draw[red, thick, ->] (0.5+6, 0.25-0.1) -- (1+6, 0.25-0.1);
            \draw[red, thick, ->] (0.5+6, 0.25-0.1) -- (0.5+6, 1);
            \filldraw[white] (6.2,0) circle (1pt);
            \draw[red] (6.2, 0) circle (1pt);
            \filldraw[white] (-0.2+6,-0.2) circle (1pt);
            \draw[red] (-0.2+6, -0.2) circle (1pt);
            \filldraw[red] (0.5+6, 0.25-0.1) circle (1pt);
            \node at (6, -1) {\(c_3\)};
        \end{tikzpicture}
    \end{center} The slope-sensitive fan \(\tilde{\Sigma}_{\upalpha}\) was given in Example \ref{slopeSen}. Consider lifts \(\tilde{c}_i \in \mathcal{C}(\Gamma, \tilde{\Sigma}_{\upalpha}, \upalpha)\) for each \(c_i\). Note that these may not be unique.
    
    
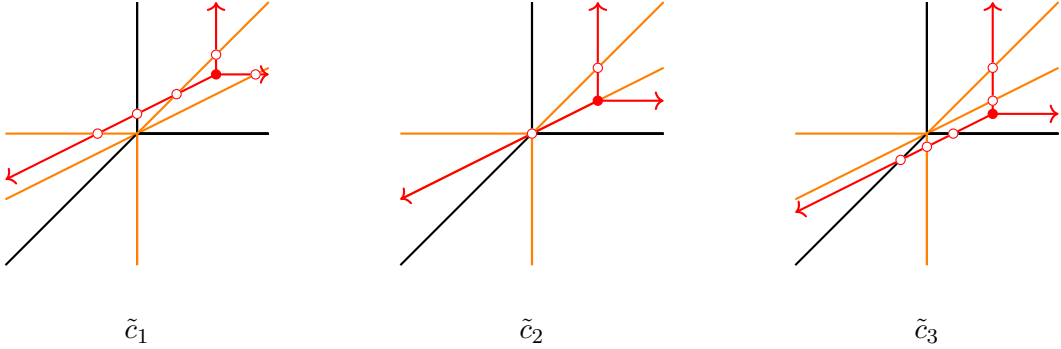
\begin{center}
        \begin{tikzpicture}[scale = 1.75]
            \draw[orange, thick] (1, 0.5) -- (-1, -0.5);
            \draw[orange, thick] (1+3, 0.5) -- (-1+3, -0.5);
            \draw[orange, thick] (1+6, 0.5) -- (-1+6, -0.5);
    
            \draw[thick] (0,0) -- (0, 1) -- (0,0) -- (1,0) -- (0,0) -- (-1,-1);
            \draw[thick] (0+3,0) -- (0+3, 1) -- (0+3,0) -- (1+3,0) -- (0+3,0) -- (-1+3,-1);
            \draw[thick] (0+6,0) -- (0+6, 1) -- (0+6,0) -- (1+6,0) -- (0+6,0) -- (-1+6,-1);
            \draw[orange, thick] (-1, 0) -- (0,0) -- (0, -1) -- (0,0) -- (1,1);
            \draw[orange, thick] (-1+3, 0) -- (0+3,0) -- (0+3, -1) -- (0+3,0) -- (1+3,1);
            \draw[orange, thick] (-1+6, 0) -- (0+6,0) -- (0+6, -1) -- (0+6,0) -- (1+6,1);
            \filldraw[red] (.5+0.1, 0.25+0.2) circle (1pt);
            
            \draw[red, thick, ->] (0.5+0.1, 0.25+0.2) -- (-1, -0.35);
            \draw[red, thick, ->] (0.5+0.1, 0.25+0.2) -- (1, 0.25+0.2);
            \draw[red, thick, ->] (0.5+0.1, 0.25+0.2) -- (0.5+0.1, 1);
            
            \filldraw[white] (0,0.15) circle (1pt);
            \draw[red] (0, 0.15) circle (1pt);
            \node at (0, -1.5) {\(\tilde{c}_1\)};

            \filldraw[red] (.5+3, 0.25) circle (1pt);
            \draw[red, thick, ->] (0.5+3, 0.25) -- (-1+3, -0.5);
            \draw[red, thick, ->] (0.5+3, 0.25) -- (1+3, 0.25);
            \draw[red, thick, ->] (0.5+3, 0.25) -- (0.5+3, 1);
            \filldraw[white] (3,0) circle (1pt);
            \draw[red] (3, 0) circle (1pt);
            \node at (3, -1.5) {\(\tilde{c}_2\)};

            \draw[red, thick, ->] (0.5+6, 0.25-0.1) -- (-1+6, 0.15-0.75);
            \draw[red, thick, ->] (0.5+6, 0.25-0.1) -- (1+6, 0.25-0.1);
            \draw[red, thick, ->] (0.5+6, 0.25-0.1) -- (0.5+6, 1);
            \filldraw[white] (6.2,0) circle (1pt);
            \draw[red] (6.2, 0) circle (1pt);
            \filldraw[white] (-0.2+6,-0.2) circle (1pt);
            \draw[red] (-0.2+6, -0.2) circle (1pt);
            \filldraw[red] (0.5+6, 0.25-0.1) circle (1pt);
            \filldraw[white] (6,-0.1) circle (1pt);
            \draw[red] (6, -0.1) circle (1pt);
            
            \node at (6, -1.5) {\(\tilde{c}_3\)};


            \filldraw[white] (0.6, 0.6) circle (1pt);
            \draw[red] (0.6, 0.6) circle (1pt);
            \filldraw[white] (0.3, 0.3) circle (1pt);
            \draw[red] (0.3, 0.3) circle (1pt);
            \filldraw[white] (-0.3, 0) circle (1pt);
            \draw[red] (-0.3, 0) circle (1pt);

            \filldraw[white] (.5+3, 0.5) circle (1pt);
            \draw[red] (0.5+3, 0.5) circle (1pt);
            \filldraw[white] (0.9, 0.45) circle (1pt);
            \draw[red] (0.9, 0.45) circle (1pt);

            \filldraw[white] (6.5, 0.5) circle (1pt);
            \draw[red] (6.5, 0.5) circle (1pt);
            \filldraw[white] (6.5, 0.25) circle (1pt);
            \draw[red] (6.5, 0.25) circle (1pt);

        \end{tikzpicture}
    \end{center}

    In each instance, the semi-stable graph \(\Gamma^{\tilde{c}_i}\) is a subdivision of \(\Gamma^{c_i}\). For \(c_2\) and \(c_3\), there is a unique realisable lift. In the case of \(c_2\), this is the realisable combinatorial type \(\tilde{c}_2\) in \(\,\overbar{\mathcal{M}}(X_{\alsig}, \upalpha)\) for which \(\sigma_v = \cone(\colVec{2}{1})\) and for \(c_3\), the type \(\tilde{c_3}\) where \(\sigma_v = \cone(\colVec{2}{1}, \colVec{1}{0})\). However, \(c_1\) has three distinct realisable lifts: 
    
    
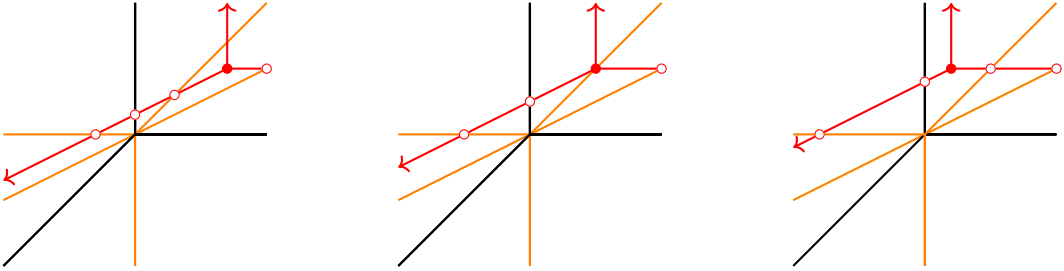
\begin{center}
        \begin{tikzpicture}[scale = 1.75]
            \draw[orange, thick] (1, 0.5) -- (-1, -0.5);
            \draw[orange, thick] (1+3, 0.5) -- (-1+3, -0.5);
            \draw[orange, thick] (1+6, 0.5) -- (-1+6, -0.5);
            \draw[thick] (0,0) -- (0, 1) -- (0,0) -- (1,0) -- (0,0) -- (-1,-1);
            \draw[thick] (0+3,0) -- (0+3, 1) -- (0+3,0) -- (1+3,0) -- (0+3,0) -- (-1+3,-1);
            \draw[thick] (0+6,0) -- (0+6, 1) -- (0+6,0) -- (1+6,0) -- (0+6,0) -- (-1+6,-1);
            \draw[orange, thick] (-1, 0) -- (0,0) -- (0, -1) -- (0,0) -- (1,1);
            \draw[orange, thick] (-1+3, 0) -- (0+3,0) -- (0+3, -1) -- (0+3,0) -- (1+3,1);
            \draw[orange, thick] (-1+6, 0) -- (0+6,0) -- (0+6, -1) -- (0+6,0) -- (1+6,1);
            \filldraw[red] (.2+6, 0.5) circle (1pt);
            \draw[red, thick, ->] (0.2+6, 0.5) -- (0.2+6, 1);
            \draw[red, thick] (0.2+6, 0.5) -- (1+6, 0.5);
            \draw[red, thick, ->] (0.2+6, 0.5) -- (-1+6, -0.1);

            \filldraw[white] (0+6,0.4) circle (1pt);
            \draw[red] (0+6, 0.4) circle (1pt);
            \filldraw[white] (0.5+6,0.5) circle (1pt);
            \draw[red] (0.5+6, 0.5) circle (1pt);
            \filldraw[white] (-0.8+6,0) circle (1pt);
            \draw[red] (-0.8+6, 0) circle (1pt);
            \filldraw[white] (1+6,0.5) circle (1pt);
            \draw[red] (1+6, 0.5) circle (1pt);

            \draw[thick] (0+3,0) -- (0+3,1) -- (0+3,0) -- (1+3,0) -- (0+3,0) -- (-1+3, -1);
            \filldraw[red] (.5+3, 0.5) circle (1pt);
            \draw[red, thick, ->] (0.5+3, 0.5) -- (0.5+3, 1);
            \draw[red, thick] (0.5+3, 0.5) -- (1+3, 0.5);
            \draw[red, thick, ->] (0.5+3, 0.5) -- (-1+3, -0.25);
            \filldraw[white] (1+3,0.5) circle (1pt);
            \draw[red] (1+3, 0.5) circle (1pt);
            \filldraw[white] (0+3,0.25) circle (1pt);
            \draw[red] (0+3, 0.25) circle (1pt);
            \filldraw[white] (-0.5+3,0) circle (1pt);
            \draw[red] (-0.5+3, 0) circle (1pt);
            
            \draw[thick] (0,0) -- (0,1) -- (0,0) -- (1,0) -- (0,0) -- (-1 , -1);
            \draw[red, thick, ->] (0.7 , 0.5) -- (0.7 , 1);
            \draw[red, thick] (0.7 , 0.5) -- (1 , 0.5);
            \draw[red, thick, ->] (0.7 , 0.5) -- (-1 , -0.35);
            \filldraw[white] (1 ,0.5) circle (1pt);
            \draw[red] (1 , 0.5) circle (1pt);
            \filldraw[white] (0 ,0.15) circle (1pt);
            \draw[red] (0 , 0.15) circle (1pt);
            \filldraw[white] (-0.3 ,0) circle (1pt);
            \draw[red] (-0.3 , 0) circle (1pt);
            \filldraw[white] (0.3 ,0.3) circle (1pt);
            \draw[red] (0.3 , 0.3) circle (1pt);
            \filldraw[red] (0.7 , 0.5) circle (1pt);

        \end{tikzpicture}
    \end{center}
   
   \begin{flushright}\emph{\textbf{End of example.}}\end{flushright}
\end{example}

    \subsection{Proof part 3: constructing a realisable combinatorial type}\label{construct}

In this section, given \(\upalpha\), \(\upalpha' \in N_0^n\) such that \(\upalpha \sim_\Sigma \upalpha'\), we describe how to build the bijection \[\mathcal{C}(\Gamma, \Sigma, \upalpha) \longleftrightarrow \mathcal{C}(\Gamma, \Sigma, \upalpha')\] for a fixed \(\Gamma \in \Gamma_{0,n}\). Beginning with a realisable combinatorial type \(c \in \mathcal{C}(\Gamma, \Sigma, \upalpha)\), we will first choose a realisable lift \(\tilde{c} \in f^{-1}(c) \subseteq \mathcal{C}(\Gamma, \tilde{\Sigma}_{\upalpha}, \upalpha)\) with semi-stable graph \(\Gamma^{\tilde{c}}\). Such a lift exists by Proposition \ref{realLift}. Given \(\tilde{c}\), in Lemma \ref{slope} we give conditions to build a tropical map \(g\colon \tilde{\Gamma} \to \tilde{\Sigma}_{\upalpha'}\) sending each \(v \in V(\tilde{\Gamma})\) and \(e \in E(\tilde{\Gamma})\) to the analogous cones of \(\sigma_v\) and \(\sigma_e\) respectively. This will then induce a combinatorial type \(\tilde{c}' \in \mathcal{C}(\Gamma, \tilde{\Sigma}_{\upalpha'}, \upalpha')\) matching the data of \(\tilde{c}\).

We begin with the following lemma. 

\begin{lemma} \label{slope}
    Let \(\upalpha \in N_0^n\) and \(\Sigma\) be a fan with corresponding slope sensitive fan \(\tilde{\Sigma}_{\upalpha}\). Let \(p\), \(q \in \mathcal{I}^\Sigma\) have balanced slope assignment \(m_{\upalpha}(p)\), \(m_{\upalpha}(q) \in N\). Then, for any \(n\)-marked graph \(\Gamma\), the slope \(m_{\upalpha}(\vv{e})\) of any \(\vv{e} \in \vv{E}(\Gamma)\) satisfies \(m_{\upalpha}(\vv{e}) \in \cone(m_{\upalpha}(p), -m_{\upalpha}(q))\) or \(m_{\upalpha}(\vv{e}) \in \cone(-m_{\upalpha}(p), m_{\upalpha}(q))\):

    \begin{center}
        \begin{tikzpicture}[scale = 2]
            \filldraw[purple!10] (0,0) -- (1,0) -- (1, -1) -- (0, -1) -- (0,0) -- (0, 1) -- (-1, 1) -- (-1, 0) -- (0,0);
            \filldraw (0,0) circle (0.75pt);
            \draw[<->, thick] (-1, 0) -- (1,0);
            \draw[thick, <->] (0,1) -- (0,-1);
            \node at (1.35, 0) {\(m_{\upalpha}(q)\)};
            \node at (0, 1.15) {\(m_{\upalpha}(p)\)};
            \node at (0, -1.2) {\(-m_{\upalpha}(p)\)};
            \node at (-1.4, 0) {\(-m_{\upalpha}(q)\)};
        \end{tikzpicture}
    \end{center}
\end{lemma}

\begin{proof}
    See \cite[Proposition 4.2]{rootsAndLogs}
\end{proof}

Now, let \(\Theta\) be the graph with two vertices \(v\) and \(w\) joined by a finite edge \(\vv{e} = (e,v)\) with slope \(m_{\upalpha}(\vv{e}) = \upalpha_I \in N\) for some \(I \subset [n]\). Let \(\sigma_v\), \(\sigma_w\) be two cones of \(\tilde{\Sigma}_{\upalpha}\) and \(f(v) \in \sigma_v^\circ\) be fixed.

\begin{lemma}
    The realisability of a combinatorial type \(c\) with graph \(\Theta\) and \(f(w) \in \sigma_w^\circ\) is equivalent to a conical containment involving \(m_{\upalpha}(\vv{e})\) and the slopes of the generators of \(\sigma_v\) and \(\sigma_w\). \label{coneContainmentLemma}
\end{lemma}

\begin{proof}
    We proceed by cases of dimensions of \(\sigma_v\) and \(\sigma_w\). Let \(p\), \(q \in \mathcal{I}^\Sigma\) correspond to cones \(\rho_p = \cone(m_{\upalpha}(p))\), \(\rho_q = \cone(m_{\upalpha}(q))\) of \(\tilde{\Sigma}_{\upalpha}\) such that \(\sigma = \cone(\rho_p, \rho_q)\) is a 2-dimensional cone of \(\tilde{\Sigma}_{\upalpha}\). By Lemma \ref{slope}, either \(m_{\upalpha}(\vv{e}) \in \cone(m_{\upalpha}(p), -m_{\upalpha}(q))\) or \(m_{\upalpha}(\vv{e}) \in \cone(-m_{\upalpha}(p), m_{\upalpha}(q))\). 
    
    There are four cases to consider. Note that if we did not have slope-sensitivity, the permissible slopes for \(m_{\upalpha}(\vv{e})\) in each case would be larger. 
    \begin{enumerate}
        \item  \(\dim(\sigma_v) = \dim(\sigma_w) = 1\). The picture is as follows: 
            
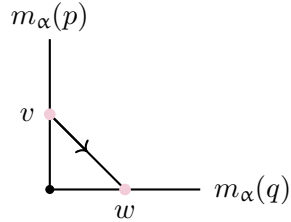
\begin{center}
                \begin{tikzpicture}[scale = 2]
                    \filldraw (0,0) circle (0.75pt);
                    \draw[thick] (0,0) -- (0,1);
                    \draw[thick] (0,0) -- (1,0);
                    \draw[thick] (0, 0.5) -- (0.5, 0);
                    \draw[thick, ->] (0, 0.5) -- (0.25, 0.25);
                    \node at (-0.15, 0.5) {\(v\)};
                    \node at (0.5, -0.15) {\(w\)};
                    \node at (0, 1.15) {\(m_{\upalpha}(p)\)};
                    \node at (1.35, 0) {\(m_{\upalpha}(q)\)};
                    \filldraw[purple!20] (0,0.5) circle (1pt);
                    \filldraw[purple!20] (0.5, 0) circle (1pt);
                \end{tikzpicture}
            \end{center} \textbf{Claim 1:} Let \(\sigma_v = \cone(m_{\upalpha}(p))\), \(\sigma_w = \cone(m_{\upalpha}(q))\). Then, \[c \text{ realisable } \iff m_{\upalpha}(\vv{e}) \in \cone^\circ(-m_{\upalpha}(p), m_{\upalpha}(q)).\] 
            \emph{Proof.} Let \(f(v) = \lambda m_{\upalpha}(p) \in \sigma_v^\circ\) for \(\lambda > 0\) be fixed.  
            
            (\(\implies\)) Assume \(m_{\upalpha}(\vv{e}) \in\cone^\circ(-m_{\upalpha}(p), m_{\upalpha}(q))\). We show that there exists a length \(\ell_e > 0\) for \(e\) such that \(f(w) \in \sigma^\circ(w)\). Let \[m_{\upalpha}(\vv{e}) = p(-m_{\upalpha}(p)) + qm_{\upalpha}(q)\] with \(p\), \(q > 0\). Then, the position of \(f(w)\) is \[\lambda m_{\upalpha}(p) + \ell_e m_{\upalpha}(\vv{e}) = \lambda m_{\upalpha}(p) + \ell_e(p(-m_{\upalpha}(p)) + q m_{\upalpha}(q)) = (\lambda - p\ell_e)m_{\upalpha}(p) + q\ell_e m_{\upalpha}(q)\] and choosing \(\ell_e = \frac{\lambda}{p} > 0\) gives \(f(w) \in \sigma^\circ(w)\). 
            
            \((\impliedby\)) Assume \(f(w) \in \sigma^\circ(w)\), i.e. \(f(w) = s m_{\upalpha}(q)\) for \(s > 0\). Then, there exists a length \(\ell_e > 0\) for \(e\) such that \(\lambda m_{\upalpha}(p) + \ell_e m_{\upalpha}(\vv{e}) = s m_{\upalpha}(q)\). Rearranging gives \[m_{\upalpha}(\vv{e}) = \frac{s}{\ell_e} m_{\upalpha}(q) + \frac{\lambda}{\ell_e} (-m_{\upalpha}(p)) \implies m_{\upalpha}(\vv{e}) \in \cone^\circ(-m_{\upalpha}(p), m_{\upalpha}(q)).\] \begin{flushright}
                \textbf{\emph{End of proof of Claim 1.}}
            \end{flushright}
                
        \item \(\dim(\sigma_v) = 1\), \(\dim(\sigma_w) = 2\). The picture is as follows: 
            
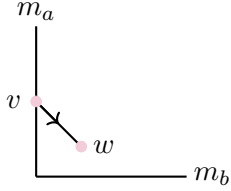
\begin{center}
                \begin{tikzpicture}[scale = 2]
                    \draw[thick] (0,0) -- (0,1);
                    \draw[thick] (0,0) -- (1,0);
                    \draw[thick] (0, 0.5) -- (0.3, 0.2);
                    \draw[thick, ->] (0, 0.5) -- (0.15, 0.35);
                    \node at (-0.15, 0.5) {\(v\)};
                    \node at (0.45, 0.2) {\(w\)};
                    \node at (0, 1.1) {\(m_a\)};
                    \node at (1.17, 0) {\(m_b\)};
                    \filldraw[purple!20] (0,0.5) circle (1pt);
                    \filldraw[purple!20] (0.3, 0.2) circle (1pt);
                \end{tikzpicture}
            \end{center}
            \textbf{Claim 2:} Let \(\sigma_v = \cone(m_a)\), \(\sigma_w = \cone(m_a, m_b)\). Then, \[c \text{ realisable } \iff m_{\upalpha}(\vv{e}) \in \cone(-m_a, m_b) \setminus \cone(-m_a).\] 
            \emph{Proof.} Let \(f(v) = \lambda m_a \in \sigma_v^\circ\) for \(\lambda > 0\) be fixed.               
            
            \((\implies)\) Assume \(m_{\upalpha}(\vv{e}) \in \cone(-m_a, m_b) \setminus \cone(-m_a)\). Let \[m_{\upalpha}(\vv{e}) = p(-m_a) + qm_b\] with \(p \geq 0\), \(q > 0\). We again show that there exists a length \(\ell_e > 0\) for \(e\) such that \(f(w) \in \sigma_w^\circ\). The position of \(f(w)\) is \[\lambda m_a + \ell_e m_{\upalpha}(\vv{e}) = \lambda m_a + \ell_e (p(-m_a) + qm_b) = (\lambda-p\ell_e) m_a + q\ell_e m_b.\] Therefore, since \(q > 0\), choosing \(\ell_e\) such that \(0 < \ell_e < \frac{\lambda}{p} \) gives \(f(w) \in \sigma^\circ(w)\).

            If \(p =0\), the position of \(f(w)\) is \[\lambda m_a + q\ell_e m_b\] and since \(\lambda\), \(q > 0\), any length \(\ell_e > 0\), will yield \(f(w) \in \sigma^\circ(w)\). 
                
            \((\impliedby)\) Assume \(f(w) \in \sigma^\circ(w)\), i.e. \(f(w) = sm_a + tm_b\) for \(s\), \(t > 0\). Then, there exists a length \(\ell_e > 0\) for \(e\) such that \[\lambda m_a + \ell_e m_{\upalpha}(\vv{e}) = s m_a + t m_b.\] Rearranging gives \[m_{\upalpha}(\vv{e}) = \frac{(s - \lambda)}{\ell_e} m_a + \frac{t}{\ell_e} m_b.\]
                
            If \(s > \lambda\), then \(m_{\upalpha}(\vv{e}) \in \cone(m_a, m_b)\) which by Lemma \ref{slope} is not possible. Else, if \(s = \lambda\), \(m_{\upalpha}(\vv{e}) \in \cone(m_b)\). Finally if \(s < \lambda\), \(m_{\upalpha}(\vv{e}) \in \cone(-m_a, m_b)\) and  \(m_{\upalpha}(\vv{e}) \not \in\cone(-m_a)\) since \(t > 0\). Combining gives \(m_{\upalpha}(\vv{e}) \in \cone(-m_a, m_b) \setminus \cone(-m_a)\). \begin{flushright}
                \textbf{\emph{End of proof of Claim 2.}} \end{flushright}
                    
        \item \(\dim(\sigma_v) = \dim(\sigma_w) = 2\). In this case, \(\sigma_v = \sigma_w = \cone(m_a, m_b)\) (else \(\vv{e}\) would cross a ray of \(\tilde{\Sigma}_{\upalpha}\)) and \(m_{\upalpha}(\vv{e})\) can be anywhere in either \(\cone(m_a, -m_b)\) or \(\cone(-m_a, m_b)\) since \(\ell_e\) can always be made small enough for \(f(v)\), \(f(w) \in \sigma_v^\circ\).  
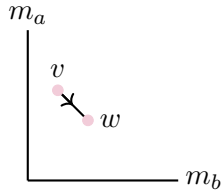
            \begin{center}
                \begin{tikzpicture}[scale = 2]
                    \draw[thick] (0,0) -- (0,1);
                    \draw[thick] (0,0) -- (1,0);
                    \draw[thick] (0.2, 0.6) -- (0.4, 0.4);
                    \draw[thick, ->] (0.2, 0.6) -- (0.3, 0.5);
                    \node at (0.2, 0.73) {\(v\)};
                    \node at (0.55, 0.4) {\(w\)};
                    \node at (0, 1.1) {\(m_a\)};
                    \node at (1.17, 0) {\(m_b\)};
                    \filldraw[purple!20] (0.2,0.6) circle (1pt);
                    \filldraw[purple!20] (0.4, 0.4) circle (1pt);
                \end{tikzpicture}
            \end{center}
            
        \item \(\dim(\sigma_v) = 2\), \(\dim(\sigma_w) = 1\). The final picture is as follows: 
            
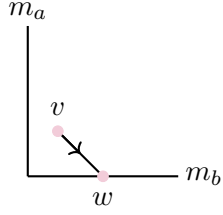
\begin{center}
                \begin{tikzpicture}[scale = 2]
                    \draw[thick] (0,0) -- (0,1);
                    \draw[thick] (0,0) -- (1,0);
                    \draw[thick] (0.2, 0.3) -- (0.5, 0);
                    \draw[thick, ->] (0.2, 0.3) -- (0.35, 0.15);
                    \node at (0.2, 0.45) {\(v\)};
                    \node at (0.5, -0.15) {\(w\)};
                    \node at (0, 1.1) {\(m_a\)};
                    \node at (1.17, 0) {\(m_b\)};
                    \filldraw[purple!20] (0.2,0.3) circle (1pt);
                    \filldraw[purple!20] (0.5, 0) circle (1pt);
                \end{tikzpicture}
            \end{center}
        \textbf{Claim 3:}  Let \(\sigma_v = \cone(m_a, m_b)\), \(\sigma_w = \cone(m_b)\). Then, \[c \text{ realisable } \iff m_{\upalpha}(\vv{e}) \in \cone(-m_a, m_b) \setminus \cone(m_b).\] 
        \emph{Proof:} Let \(f(v) = \lambda m_a + \mu m_b \in \sigma_v^\circ\) for \(\lambda\), \(\mu > 0\) be fixed.

        \((\implies)\) Let \(m_{\upalpha}(\vv{e}) = p(-m_a) + qm_b\) with \(p > 0\), \(q \geq 0\). We show that there exists a length \(\ell_e > 0\) for \(e\) such that \(f(w) \in \sigma^\circ(w)\). The position of \(f(w)\) is \[\lambda m_a + \mu m_b+ \ell_e m_{\upalpha}(\vv{e}) = \lambda m_a + \mu m_b + \ell_e(p(-m_a) + q m _b) = m_a (\lambda -p \ell_e) + m_b(\mu + q\ell_e).\] Choosing \(\ell_e = \frac{\lambda}{p} > 0 \) gives \(f(w) \in\cone^\circ(m_b)\). 
        
        \((\impliedby)\) Let \(f(w) \in \sigma^\circ(w)\), say \(f(w) = \nu m_b\) for \(\nu > 0\). Then, there exists a length \(\ell_e > 0\) for \(e\) such that \[\lambda m_a + \mu m_b + \ell_e m_{\upalpha}(\vv{e}) = \nu m_b \implies m_{\upalpha}(\vv{e}) = \frac{\nu-\mu}{\ell_e} m_b - \frac{\lambda}{\ell_e} m_a.\] If \(\nu < \mu\), then \(m_{\upalpha}(\vv{e}) \in \cone(-m_a, -m_b)\) which again by Lemma \ref{slope} not possible. If \(\nu = \mu\), we have \(m_{\upalpha}(\vv{e}) \in \cone(-m_a)\), and if \(\nu > \mu\), \(m_{\upalpha}(\vv{e}) \in \cone(m_b, -m_a)\) with \(m_{\upalpha}(\vv{e}) \not \in \cone(m_b)\) since \(\lambda > 0\). Combining both possibilities gives the result.  \begin{flushright} \textbf{\emph{End of proof of Claim 3.}} \end{flushright}
    \end{enumerate}
    All cases have been considered with each being equivalent to a conical containment involving \(m_{\upalpha}(\vv{e})\), \(m_a\) and \(m_b\).
\end{proof}

\begin{lemma}
    Let \(\upalpha\), \(\upalpha' \in N_0^n\) with \(\upalpha \sim_\Sigma \upalpha'\) and \(\Sigma\) in \(N_\R\) a complete fan. Consider \(\tilde{\Sigma}_{\upalpha}\), \(\tilde{\Sigma}_{\upalpha'}\). Let \(f\colon \Gamma \to \tilde{\Sigma}_{\upalpha}\) be a tropical map for an \(n\)-marked graph \(\Gamma\). Let \(\ell\) be a leg of \(\Gamma\) with supporting vertex \(v\) and \(f(v) \in \sigma_{v}^\circ\) be such that \(f(\ell)\) does not transversely intersect any ray of \(\tilde{\Sigma}_{\upalpha}\). Then, any tropical map \(g\colon \Gamma \to \tilde{\Sigma}_{\upalpha'}\) such that \(g(v) \in (\sigma_{v}')^\circ\) is such that \(g(\ell)\) does not intersect any ray of \(\tilde{\Sigma}_{\upalpha'}\). \label{crossLeg}
\end{lemma}

\begin{proof}
    Let \(\ell\) have slope \(m_{\upalpha}(\ell)\), \(\rho_{\ell} = \cone(m_{\upalpha}(\ell))\) and \(\rho_1\), \(\rho_2 \in \tilde{\Sigma}_{\upalpha}(1)\) be two rays such that \(\sigma_1 = \cone(\rho_1, \rho_{\ell})\), \(\sigma_2 = \cone(\rho_2, \rho_{\ell})\) are 2-dimensional cones of \(\tilde{\Sigma}_{\upalpha}\). Let \(\sigma \in \Sigma\) be the smallest cone such that \(f(\ell) \subseteq \sigma\). Then, \(f(\ell)\) has no intersections with rays of \(\tilde{\Sigma}_{\upalpha}\) if and only if \(m_{\upalpha}(\ell) \in \sigma\). As such, we have \(\sigma = \rho_\ell\), \(\sigma_1\) or \(\sigma_2\). Thus, we have either \(f(v) \in \sigma_1\) or \(\sigma_2\) (including the possibility \(f(v) \in \rho_1\), \(\rho_2\) or \(o\)). The image is as follows:

        
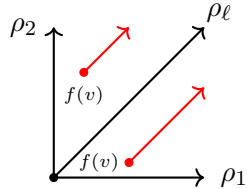
\begin{center}
         \begin{tikzpicture}[scale = 2]
            \filldraw (0,0) circle (0.75pt);
             \draw[thick, ->] (0,0) -- (0,1);
             \draw[thick, ->] (0,0) -- (1,0);
             \draw[thick, ->] (0,0) -- (1,1);
             \node at (1.2, 0) {\(\rho_1\)};
             \node at (1.1, 1.1) {\(\rho_{\ell}\)};
             \node at (-0.2, 1) {\(\rho_2\)};
             \filldraw[red] (0.2, 0.7) circle (0.75pt);
             \filldraw[red] (0.5, 0.1) circle (0.75pt);
             \draw[red, thick, ->] (0.2, 0.7) -- (0.5, 1);
             \draw[red, thick, ->] (0.5, 0.1) -- (1, 0.6);
             \node at (0.3, 0.1) {\tiny{\(f(v)\)}};
             \node at (0.2, 0.55) {\tiny{\(f(v)\)}};
         \end{tikzpicture}
     \end{center}
    
    Now, by Lemma \ref{analogousCone}, the analogous cones \(\sigma_1'\) and \(\sigma_2'\) are both cones of \(\tilde{\Sigma}_{\upalpha'}\). In each case, any map \(g\colon \Gamma \to \tilde{\Sigma}_{\upalpha'}\), mapping \(g(v)\) to the analogous cone will again yield no intersections of \(\ell\) with the rays of \(\tilde{\Sigma}_{\upalpha'}\) since in each case, \(m_{\upalpha'}(\ell) \in \sigma'\).   \end{proof}

We give the final proposition before the proof of Theorem \ref{thm: body}. 

\begin{proposition}
    Let \(\upalpha\), \(\upalpha' \in N_0^n\) with \(\upalpha \sim_\Sigma \upalpha'\). Then, there is a bijection \begin{align*}
        \mathcal{C}(\Gamma, \Sigma, \upalpha) &\longleftrightarrow  \mathcal{C}(\Gamma, \Sigma, \upalpha') \\ 
        c & \longleftrightarrow c'
    \end{align*} such that \([\mathcal{M}(c)] = [\mathcal{M}(c')] \in K_0(\mathrm{Var}/\mathbb{C})\).  \label{bigProp}
\end{proposition}

\begin{proof}
    Let \(c\) have semi-stable graph \(\Gamma\). By Lemma \ref{realLift}, we can choose a realisable lift \(\tilde{c} \in \mathcal{C}(\Gamma, \tilde{\Sigma}_{\upalpha}, \upalpha)\) with semi-stable graph \(\tilde{\Gamma}\). Consider now \(\upalpha'\), \(\tilde{\Sigma}_{\upalpha'}\) and a metric enhancement \(\vsqsubset_{\tilde{\Gamma}}\). We define a tropical map \[g\colon \sqC_{\tilde{\Gamma}} \to \tilde{\Sigma}_{\upalpha'}\] as follows. Choose a starting vertex \(v \in V(\tilde{\Gamma})\) with cone assignment \(\tilde{\sigma}_v\) in \(\tilde{c}\) and map this arbitrarily to the analogous cone \(\tilde{\sigma}_v' \in \tilde{\Sigma}_{\upalpha'}\) (which is a cone of \(\tilde{\Sigma}_{\upalpha'}\) by Lemma \ref{analogousCone}). For each neighbouring vertex \(w\) of \(v\), choose a length for the finite edge \(\vv{e}\) joining \(v\) to \(w\) such that \(g(w) \in (\tilde{\sigma}_w')^\circ\). By Lemma \ref{coneContainmentLemma}, the existence of such a length is equivalent to a conical containment involving \(m_{\upalpha}(\vv{e})\) and the slopes of the rays of \(\tilde{\sigma}_v'\) and \(\tilde{\sigma}_w'\). Since \(\tilde{c}\) is realisable, this conical containment is satisfied by \(\upalpha\), and hence by Lemma \ref{coneConditions}, it is also satisfied by \(\upalpha'\). 

    Continue this process for all neighbouring vertices of \(v\) and the rest of \(V(\tilde{\Gamma})\) to get a tropical map \(g\colon \vsqsubset_{\tilde{\Gamma}} \to \tilde{\Sigma}_{\upalpha'}\). Note that there is no dependence on previous lengths as Lemma \ref{coneContainmentLemma} holds for arbitrary starting positions. Take the combinatorial type of \(g\) as per Construction \ref{construction} to obtain \[\tilde{c}' \in \mathcal{C}(\Gamma, \tilde{\Sigma}_{\upalpha'}, \upalpha').\]

    The semi-stable graph associated to \(\tilde{c}'\) is the same as \(\tilde{\Gamma}\). By assuming each \(v\) and \(w\) are neighbours, since \(\upalpha \sim_\Sigma \upalpha'\), there will be no further subdivision of \(\vv{e}\) as in each case, either \(\tilde{\sigma}'_v = \tilde{\sigma}'_w\) or one is a face of the other. This holds for any finite edge of \(\tilde{\Gamma}\). There will also be no further subdivisions of any leg of \(\tilde{\Gamma}\), by Lemma \ref{crossLeg}.  
    
    Finally, apply the forgetful map \[f\colon \mathcal{C}(\Gamma, \tilde{\Sigma}_{\upalpha'}, \upalpha') \to \mathcal{C}(\Gamma, \Sigma, \upalpha')\] to obtain \(c' \in  \mathcal{C}(\Gamma, \Sigma, \upalpha')\) with semi-stable graph \(\Gamma'\). By the preceding construction, \(c'\) is realisable and \(\Gamma'\) is a subdivision of \(\Gamma\) at the legs or edges which intersect 0 or 1-dimensional cones of \(\Sigma\). The number and location of these intersections is the same for \(c\) and \(c'\), hence \(\Gamma' = \Gamma\) and for all \(v \in V(\Gamma)\), \(e \in E(\Gamma)\), the cone assignments match for \(c\) and \(c'\). Since Formula \eqref{form} depends only upon this data, we have \([\mathcal{M}(c)] = [\mathcal{M}(c')] \in K_0(\mathrm{Var}/\mathbb{C})\). 

    Reversing the roles of \(\upalpha\) and \(\upalpha'\) gives that this is a bijection.
\end{proof}

\begin{proof}[Proof of Theorem \ref{thm: body}]
    Let \(\overbar{\mathcal{M}}(X_\Sigma, \upalpha)\) and \(\overbar{\mathcal{M}}(X_\Sigma, \upalpha')\) have stratifications  \[\overbar{\mathcal{M}}(X_\Sigma, \upalpha) = \bigsqcup_{c}\mathcal{M}(c), \quad \quad \overbar{\mathcal{M}}(X_\Sigma, \upalpha') = \bigsqcup_{c'} \mathcal{M}(c')\] by realisable combinatorial types \(c\), \(c'\). By Proposition \ref{bigProp}, there exists a bijection \(c \leftrightarrow c'\) such that \([\mathcal{M}(c)] = [\mathcal{M}(c')] \in K_0(\mathrm{Var}/\C)\). Therefore, by Proposition \ref{splitter} \([\,\overbar{\mathcal{M}}(X_\Sigma, \upalpha)] = [\,\overbar{\mathcal{M}}(X_\Sigma, \upalpha')] \in K_0(\mathrm{Var}/\C)\).  \end{proof}

    \subsection{Elaborations and generalisations}\label{other}

\subsubsection{Wall and chamber decomposition}\label{walls} 

We now give hypersurfaces in \(N^n\) which form the walls of the \(\Sigma\)-slope decomposition. Each wall is given by either a linear or quadratic equation. 

As \(\upalpha\) moves within a chamber, the cyclic ordering of \(\mathcal{I}^\Sigma\), as defined in Definition \ref{cyclicOrderDef}, will remain constant. 
Crossing a wall will change this cyclic ordering. We allow \(\upalpha\) to sit within one (or more) of these walls: this is when \(\upalpha_I = k\upalpha_{J}\) for some subsets \(I\), \(J\) of \([n]\) and \(k\in \mathbb{Q}^*\), or when \(\upalpha_I\) becomes parallel to a ray of \(\Sigma\). 

We restrict \(\upalpha\) to be within a special subset of \(N^n_0\) from the onset, similar to the resonance chambers in \cite{hurw}. For \(I \subset[n]\), define \[W_I := \left\{\upalpha \in N^n : \upalpha_I = \colVec{0}{0}\right\}.\] Note that \(N_0^n = W_{[n]}\). For Theorem \ref{thm: body}, as described, we require \[\upalpha \in W_{[n]} \setminus \bigcup_{\substack{I \subset [n]\\I \neq \emptyset}} W_I.\] We will discuss how to adapt the definition of \(\Sigma\)-equivalence and adjust the proof of Theorem \ref{thm: body} when \(\upalpha\) lands in one (or more) \(W_I\) in Subsection \ref{zero}. 


\begin{definition}[Quadratic walls of the \(\Sigma\)-slope decomposition]
    Let \(I\), \(J \subset [n]\) be non-empty and distinct. Define quadratic hypersurfaces \[W_{IJ} = \left\{(\colVec{x_1}{y_1}, \ldots, \colVec{x_n}{y_n}) \in N^n : \sum_{i \in I} x_i \sum_{j \in J} y_j = \sum_{i \in I} y_i \sum_{j \in J} x_j\right\}.\] 
\end{definition}

\begin{definition}[Linear walls of the \(\Sigma\)-slope decomposition] Let \( \rho \in \Sigma(1)\) with primitive generator \(v_{\rho} = \colVec{{\rho}_x}{{\rho}_y}\). For \(I \subset [n]\) non-empty, define \[W_{\rho, I} = \left\{(\colVec{x_1}{y_1}, \ldots, \colVec{x_n}{y_n}) \in N^n :  \rho_x \sum_{i \in I} y_i = \rho_y \sum_{i \in I} x_i\right\}.\] For all \(I\), \(J \subset [n]\), \(\rho \in \Sigma(1)\), the \(W_{IJ}\) and \(W_{\rho, I}\) form the walls of the \(\Sigma\)-slope decomposition. 
\end{definition}

\begin{lemma}
    As \(\upalpha \in N^n_0\) varies in the chambers of the \(\Sigma\)-slope decomposition, the cyclic ordering of \(\nu_{\upalpha}(\mathcal{I}^\Sigma)\) is constant.
\end{lemma}

\begin{proof}
    For \(p\), \(q\), \(r \in \mathcal{I}^\Sigma\), let \([p, q, r]\) hold. Without loss of generality, let \(m_{\upalpha}(p)\), \(m_{\upalpha}(q)\), \(m_{\upalpha}(r)\) be contained within a half space. Then, a change of cyclic ordering is, without loss of generality, \([q, p, r]\). If \(q = I\), \(p = J \subset [n]\), the cyclic ordering will change at the point when \(m_{\upalpha}(I) = k m_{\upalpha}(J)\) for \(k \in \mathbb{Q}^*\), i.e. when \[\sum_{i \in I}x_i \sum_{j \in J} y_j = \sum_{i \in I} y_i \sum_{j \in J} x_j.\] This is when \(\upalpha\) crosses \(W_{IJ}\). 

    If instead \(p = \rho \in \Sigma(1)\) and \(q = I \subset [n]\), the cyclic ordering changes when \(m_{\upalpha}(\rho) = k v_\rho\), i.e. when \[\rho_y\sum_{i \in I} x_i= \rho_x \sum_{i \in I} y_i.\] It cannot be the case that both \(q\) and \(p\) correspond to rays of \(\Sigma\) as these are fixed and so cannot change orders. 
\end{proof}

\subsubsection{Removing the non-zero slope assumption} \label{zero}

If there exists \(I \subset [n]\) for which \(\upalpha_I = \colVec{0}{0}\), we can still make sense of the cyclic ordering by using Definition \ref{cyclicOrderDef} with the set \(\mathcal{I}^\Sigma \setminus \{I, I^{\mathrm{c}}\}\). If \(\upalpha_I\) vanishes for multiple subsets, we simply remove all of these (and their complements) from \(\mathcal{I}^\Sigma\). In these instances, for \(\upalpha\) and \(\upalpha'\) to be considered to be in the ``same chamber'', they must both fall into the same walls, say \(\cap_{i = 1}^k W_{I_i}\) and induce the same cyclic ordering on \(\mathcal{I}^\Sigma \setminus \{I_1, I_1^{\mathrm{c}}, \ldots, I_k, I_k^{\mathrm{c}}\}\). 

A similar procedure can be done to generalise \cite[Theorem B]{kannan}. Note that we will now use notation and terminology from this paper without explanation. This theorem holds for \(\upalpha \in \Z^n\) away from the walls of the resonance chamber.  However, one can consider \(\upalpha \in W_I\) for some \(I \subset [n]\) to still move within a chamber of the resonance decomposition. 

To find combinatorial descriptions for \([\mathcal{M}_{\upalpha}(\mathbf{T})]\) when \(\upalpha \in W_I\), we make some adjustments to the set-up. We first modify the definition of an \(\upalpha\)-directing of a tree. The previous definition does not hold when \(\upalpha \in W_I\) since finding an \(\upalpha\)-directing relies on there being no proper subset \(I\) such that \(\upalpha_I = 0\). To handle this, if there now exists an internal edge \(e\) such that \(w(T_1) = w(T_2) = 0\), we contract \(e\). We say that this edge has \emph{weight 0} in the dual graph of \(\mathbf{T}\). We define an \(\upalpha\)-directing of the contracted graph, and associate this with our original source curve \(\mathbf{T}\) and proceed as before. Thus, the \(\upalpha\)-directing is still constant on chambers since to pass from one directing to another involves crossing a wall of the form \(\upalpha_J = 0\). 

\begin{example}
    Let \(\upalpha = (3, -2, 2, 4, -7)\). Then, \(\upalpha \in W_{\{23\}}\). We contract the edge in red to arrive at a curve which we can give a \(\upalpha\)-directing to. 
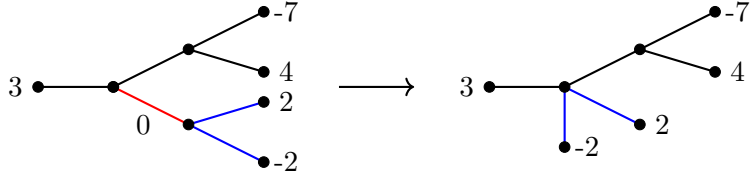
\begin{center}
    \begin{tikzpicture}[scale = 2]
        \draw[thick] (0,0) -- (0.5,0.25);
        \draw[red, thick] (0,0) -- (0.5, -0.25);
        \draw[thick] (0,0) -- (-0.5, 0);
        \node at (-0.65, 0) {3};
        \node at (1.15, 0.1) {4};
        \node at (1.15, 0.5) {-7};
        \draw[thick] (0.5, 0.25) -- (1, 0.5);
        \draw[thick] (0.5, 0.25) -- (1, 0.1);
        \draw[blue, thick] (0.5, -0.25) -- (1, -0.5);
        \draw[blue, thick] (0.5, -0.25) -- (1, -0.1);
        \node at (1.15, -0.1) {2};
        \node at (1.15, -0.5) {-2};
        \node at (0.2, -0.25) {0};
        \draw[thick, ->] (1.5, 0) -- (2,0);
        \filldraw (0,0) circle (1pt);
        \filldraw (0.5,.25) circle (1pt);
        \filldraw  circle (1pt);
        \filldraw (0.5,-0.25) circle (1pt);
        \filldraw (-0.5,0) circle (1pt);
        \filldraw (1, 0.5) circle (1pt);
        \filldraw (1, -0.5) circle (1pt);
        \filldraw (1,0.1) circle (1pt);
        \filldraw (1,-0.1) circle (1pt);

        \draw[thick] (3,0) -- (3.5,0.25);
        \draw[blue, thick] (3,0) -- (3.5, -0.25);
        \draw[thick] (3,0) -- (2.5, 0);
        \filldraw (3.5,0.25) circle (1pt);
        \filldraw (3.5,-0.25) circle (1pt);
        \filldraw (2.5,0) circle (1pt);
        \node at (2.35, 0) {3};
        \draw[thick] (3.5, 0.25) -- (4, 0.5);
        \draw[thick] (3.5, 0.25) -- (4, 0.1);
        \filldraw (4, 0.5) circle (1pt);
        \filldraw (4,0.1) circle (1pt);
        \draw[blue, thick] (3, 0) -- (3, -0.4);
        \filldraw (3, -0.4) circle (1pt);
        \filldraw (3,0) circle (1pt);
        \node at (4.15, 0.5) {-7};
        \node at (4.15, 0.1) {4};
        \node at (3.65,-0.25) {2}; 
        \node at (3.15, -0.4) {-2};

    \end{tikzpicture}
    \begin{flushright}\emph{\textbf{End of example.}}\end{flushright}
\end{center}
\end{example}

One must also make an adjustment to the strata formula given in \cite[Equation 2.14]{kannan}, reducing the power of \(\C^*\) appearing. Given a vertex \(v \in V(\mathbf{P})\), define the following number: \[\lambda_v = |\{E(\mathbf{T}) \ni e \in \pi^{-1}(v) : w(e) = 0\}|.\] An edge of weight zero in the dual graph indicates, in the algebro-geometric picture, two curve components that are stacked above each other and meet at a node with slope 0. Thus, the factor of \(\C^*\) coming from each curve component is reduced due to matching at the node. Adjusting for this gives the following isomorphism of stacks:
\begin{equation*}
    \mathcal{M}_{\upalpha}(\pi: \mathbf{T} \to \mathbf{P}) \cong \prod_{\substack{v \in V(\mathbf{T}) \\ \text{val}(v) \geq 3}} \mathcal{M}_{0, \text{val}(v)} \times \prod_{\substack{v \in V(\textbf{T}) \\ \text{val}(v) = 2}} B \mu_{w(v)} \times \prod_{v \in V(\textbf{P})} (\C^*)^{|\{u \in \pi^{-1}(v) | \text{val}(u) \geq 3\}| -1 - \lambda_v}.
\end{equation*}

   \subsubsection{Other topological invariants} \label{satake}

Given a chamber structure where the class of our space in the Grothendieck ring is constant, we also get that the topological Euler characteristic is constant, and other invariants which satisfy the cut and paste relation. However, one can also consider invariants such as the orbifold Euler characteristic (also known as the Euler--Satake characteristic). This was first introduced by Satake \cite[Section 3]{satake}. 

In our case, the orbifold Euler characteristic is not constant on the chamber decomposition we have defined: consider \(\upalpha = \left(\colVec{r}{0}, \colVec{0}{r}, \colVec{-r}{-r}\right)\) and the following combinatorial type: 

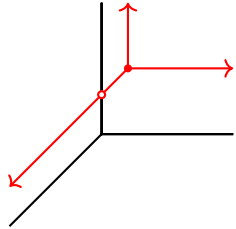
\begin{center}
    \begin{tikzpicture}[scale = 1.75]
        \draw[thick] (0,0) -- (1,0) -- (0,0) -- (0,1) -- (0,0) -- (-0.7, -0.7);
        \filldraw[red] (0.2, 0.5) circle (0.75pt);
        \draw[red, ->, thick] (0.2, 0.5) -- (1, 0.5); 
        \draw[red, ->,thick] (0.2, 0.5) -- (0.2, 1);
        \draw[red, ->,thick] (0.2, 0.5) -- (-0.7, -0.4);
        \filldraw[white] (0, 0.3) circle (0.75pt);
        \draw[thick, red] (0, 0.3) circle (0.75pt);
    \end{tikzpicture}
\end{center}

Our map on the stable component is constant: \[z \mapsto [0:0:1].\] However, the map from the unstable component has the following form: \[[S:T] \xmapsto{f} [S^r: 0: T^r].\]

Now, we notice that our source curve has an action by \(\mu_r\), the group of \(r^{\mathrm{th}}\) roots of unity, which commutes with \(f\). Thus, \(\mathrm{Aut}(f) = \mu_r\). When we take the orbifold Euler characteristic, this automorphism group will contribute \(\frac{1}{r}\). We can increase \(r\) without leaving the chamber of the \(\Sigma\)-slope decomposition, and so, the orbifold Euler characteristic is not constant on these chambers. This issue arises already in dimension 1 for the space of maps considered in \cite{kannan}. 

\printbibliography
\end{document}